\documentclass[11pt, reqno]{amsart}
\usepackage{amssymb}
\usepackage{bbm}
\usepackage{amsthm}
\usepackage{nccmath}
\usepackage[bookmarks=false]{hyperref}
\usepackage{etoolbox}
\usepackage{array}
\usepackage{xparse}
\usepackage{enumitem}
\usepackage{tikz}
\usepackage{mathtools}
\usepackage[font=scriptsize]{caption}

\calclayout

\newtheorem{thm}{Theorem}[section]

\newtheorem{lemma}[thm]{Lemma}
\newtheorem{prop}[thm]{Proposition}

\theoremstyle{definition}
\newtheorem{definition}[thm]{Definition}

\theoremstyle{remark}

\newtheorem{rmk}[thm]{Remark}

\newcommand{\thmref}[1]{Theorem~\ref{#1}}

\newcommand{\secref}[1]{\S\ref{#1}}
\newcommand{\lemref}[1]{Lemma~\ref{#1}} 
\newcommand{\propref}[1]{Proposition~\ref{#1}}

\newcommand*{\one}{\mathbbm{1}}

\newcommand*{\qq}{\qquad}

\newcommand*{\tx}[1]{\text{#1}}

\newcommand*{\ep}{\epsilon}

\newcommand*{\suchthat}{\, \middle| \,}

\newcommand*{\myoverline}[3]{\mkern -#1mu\overline{\mkern#1mu#3\mkern#2mu}\mkern -#2mu}	

\newcommand*{\zbar}{\myoverline{-2}{0}{\z}}

\newcommand*{\Omegabar}{\myoverline{0}{0}{\Omega} }

\newcommand*{\Psibar}{\myoverline{0}{0}{\Psi}}
\newcommand*{\Psizbar}{\myoverline{0}{0}{\Psi_z}}

\newcommand*{\sbar}{\myoverline{0}{0}{s}}

\newcommand*{\ybar}{\myoverline{0}{0}{y}}
\newcommand*{\zonebar}{\myoverline{0}{0}{\zone}}
\newcommand*{\ztwobar}{\myoverline{0}{0}{\ztwo}}
\newcommand*{\Xbar}{\myoverline{0}{0}{X}}

\newcommand*{\Xonebar}{\myoverline{0}{0}{X_1}}
\newcommand*{\Xtwobar}{\myoverline{0}{0}{X_2}}

\newcommand*{\Hspbar}{\myoverline{1}{1}{\Hsp}}

\newcommand*{\Omegaplusbar}{\myoverline{1}{2}{\Omega}_{+}}
\newcommand*{\Lambdabar}{\myoverline{0}{0}{\Lambda}}

\newcommand*{\half}{\frac{1}{2}}

\newcommand*{\Rsp}{\mathbb{R}}
\newcommand*{\Csp}{\mathbb{C}}
\newcommand*{\Nsp}{\mathbb{N}}

\newcommand*{\Dsp}{\mathbb{D}}

\newcommand*{\Hsp}{\mathbb{H}}
\newcommand*{\Hspplus}{\Hsp_{+}}

\newcommand*{\Omegaplus}{\Omega_{+}}

\newcommand*{\Lone}{L^1}
\newcommand*{\Ltwo}{L^2}

\newcommand*{\Linfty}{L^{\infty}}

\DeclareMathOperator*{\esssup}{ess\,sup}

\DeclareMathOperator*{\supp}{supp}
\newcommand*{\diff}{\mathop{}\! d}

\newcommand*{\compose}[1]{\circ{#1}}

\newcommand*{\Imag}{\tx{Im}}
\newcommand*{\Real}{\tx{Re}}

\newcommand*{\grad}{\nabla}

\newcommand*{\pt}{\partial_t}

\newcommand*{\btil}{\widetilde{b}}

\newcommand*{\w}{\omega}
\newcommand*{\wtil}{\widetilde{\omega}}

\newcommand*{\Psiz}{\Psi_z}

\newcommand*{\z}{z}

\newcommand*{\zone}{z_1}
\newcommand*{\ztwo}{z_2}

\newcommand*{\nobrac}[1]{ #1 }

\DeclarePairedDelimiter{\oldbrac}{\lparen}{\rparen}			
\NewDocumentCommand{\brac}{ s o m }{						
	\IfBooleanT{#1}{
  		\IfValueT{#2}{\oldbrac[#2]{#3}}
		\IfValueF{#2}{\oldbrac{#3}} 
	}
	\IfBooleanF{#1}{
  		\IfValueT{#2}{\PackageError{mypackage}{Incorrect use of brac. Insert star}{}}
		\IfValueF{#2}{\oldbrac*{#3}} 
	}		
}

\DeclarePairedDelimiter\oldcbrac{\lbrace}{\rbrace}				
\NewDocumentCommand{\cbrac}{ s o m }{					
	\IfBooleanT{#1}{
  		\IfValueT{#2}{\oldcbrac[#2]{#3}}
		\IfValueF{#2}{\oldcbrac{#3}} 
	}
	\IfBooleanF{#1}{
  		\IfValueT{#2}{\PackageError{mypackage}{Incorrect use of cbrac. Insert star}{}}
		\IfValueF{#2}{\oldcbrac*{#3}} 
	}		
}

\DeclarePairedDelimiter\oldsqbrac{\lbrack}{\rbrack}				
\NewDocumentCommand{\sqbrac}{ s o m }{					
	\IfBooleanT{#1}{
  		\IfValueT{#2}{\oldsqbrac[#2]{#3}}
		\IfValueF{#2}{\oldsqbrac{#3}} 
	}
	\IfBooleanF{#1}{
  		\IfValueT{#2}{\PackageError{mypackage}{Incorrect use of sqbrac. Insert star}{}}
		\IfValueF{#2}{\oldsqbrac*{#3}} 
	}		
}

\DeclarePairedDelimiter{\oldabs}{\lvert}{\rvert}
\NewDocumentCommand{\abs}{ s o m }{						
	\IfBooleanT{#1}{
  		\IfValueT{#2}{\oldabs[#2]{#3}}
		\IfValueF{#2}{\oldabs{#3}} 
	}
	\IfBooleanF{#1}{
  		\IfValueT{#2}{\PackageError{mypackage}{Incorrect use of abs. Insert star}{}}
		\IfValueF{#2}{\oldabs*{#3}} 
	}		
}

\DeclarePairedDelimiterX{\oldnorm}[1]{\lVert}{\rVert}{#1}
\NewDocumentCommand{\norm}{ s o o m }{					
	\IfValueT{#2} {
		\IfBooleanT{#1}{
  			\IfValueT{#3}{\oldnorm[#2]{#4}_{#3}}
			\IfValueF{#3}{\oldnorm{#4}_{#2}} 
		}
		\IfBooleanF{#1}{
  			\IfValueT{#3}{\PackageError{mypackage}{Incorrect use of norm. Insert star}{}}
			\IfValueF{#3}{\oldnorm*{#4}_{#2}} 
		}
	}
	\IfValueF{#2} {
		\IfBooleanT{#1}{\oldnorm{#4}}	
		\IfBooleanF{#1}{\oldnorm*{#4}}		
	}	
}

\makeatletter
\def\black@#1{%
    \noalign{%
        \ifdim#1>\displaywidth
            \dimen@\prevdepth
            \nointerlineskip
            \vskip-\ht\strutbox@
            \vskip-\dp\strutbox@
            \vbox{\noindent\hbox to \displaywidth{\hbox to#1{\strut@\hfill}}}%
            \prevdepth\dimen@
        \fi
    }%
}
\makeatother

\makeatletter
\renewcommand{\tocsection}[3]{%
  \indentlabel{\@ifnotempty{#2}{\bfseries\ignorespaces#1 #2\quad}}\bfseries#3}
\renewcommand{\tocsubsection}[3]{%
  \indentlabel{\@ifnotempty{#2}{\ignorespaces#1 #2\quad}}#3}

\newcommand\@dotsep{4.5}
\def\@tocline#1#2#3#4#5#6#7{\relax
  \ifnum #1>\c@tocdepth 
  \else
    \par \addpenalty\@secpenalty\addvspace{#2}%
    \begingroup \hyphenpenalty\@M
    \@ifempty{#4}{%
      \@tempdima\csname r@tocindent\number#1\endcsname\relax
    }{%
      \@tempdima#4\relax
    }%
    \parindent\z@ \leftskip#3\relax \advance\leftskip\@tempdima\relax
    \rightskip\@pnumwidth plus1em \parfillskip-\@pnumwidth
    #5\leavevmode\hskip-\@tempdima{#6}\nobreak
    \leaders\hbox{$\m@th\mkern \@dotsep mu\hbox{.}\mkern \@dotsep mu$}\hfill
    \nobreak
    \hbox to\@pnumwidth{\@tocpagenum{\ifnum#1=1\bfseries\fi#7}}\par
    \nobreak
    \endgroup
  \fi}
\AtBeginDocument{%
\expandafter\renewcommand\csname r@tocindent0\endcsname{0pt}
}
\def\l@subsection{\@tocline{2}{0pt}{2.5pc}{5pc}{}}
\makeatother

\makeatletter
 \def\@testdef #1#2#3{%
   \def\reserved@a{#3}\expandafter \ifx \csname #1@#2\endcsname
  \reserved@a  \else
 \typeout{^^Jlabel #2 changed:^^J%
 \meaning\reserved@a^^J%
 \expandafter\meaning\csname #1@#2\endcsname^^J}%
 \@tempswatrue \fi}
\makeatother

\makeatletter
\newcommand*{\rom}[1]{\expandafter\@slowromancap\romannumeral #1@}
\makeatother

\makeatletter
\patchcmd{\@sect}{\@addpunct.}{}{}{}
\patchcmd{\subsection}{-.5em}{1em}{}{}
\makeatother

\begin{document}

\title[2D Euler uniqueness]{Uniqueness of the 2D Euler equation on a corner domain with non-constant vorticity around the corner}
\author[Siddhant Agrawal]{Siddhant Agrawal$^1$}
\address{$^1$ 
Instituto de Ciencias Matem\'aticas, ICMAT, Madrid, Spain}
\email{siddhant.govardhan@icmat.es}

\author[Andrea R. Nahmod]{Andrea R. Nahmod$^2$}
\address{$^2$ 
Department of Mathematics,  University of Massachusetts,  Amherst MA 01003}
\email{nahmod@math.umass.edu}
\thanks{$^2$ {A.N. is funded in part by NSF DMS-1800852 and the Simons Foundation Collaborations Grant on Wave Turbulence (Nahmod's Award ID 651469).}} 


\begin{abstract}
We consider the 2D incompressible Euler equation on a corner domain $\Omega$ with angle $\nu\pi$ with $\half<\nu<1$. We prove that if the initial vorticity $\w_0 \in \Lone(\Omega)\cap\Linfty(\Omega)$ and if $\w_0$ is non-negative and supported on one side of the angle bisector of the domain, then the weak solutions are unique. This is the first result which proves uniqueness when the velocity is far from Lipschitz and the initial vorticity is non-constant around the boundary. 
\end{abstract}

\maketitle
\section{Introduction}

We are interested in studying incompressible fluids in two dimensions. The prototypical equation describing inviscid incompressible  flows is the Euler equation. With a view towards practical applications, we will be interested in studying the Euler equation on corner domains and more specifically the uniqueness problem for this equation. The uniqueness problem for the 2D Euler equation on obtuse angled corners  is wide open due to the fact that the velocity is not log-Lipschitz in these domains. In this paper we prove the first uniqueness result in such domains when the vorticity is non-constant around the corner, which is the main source of difficulty in proving uniqueness. To prove this result, we introduce a novel energy functional with a time dependent weight along with control of the particle trajectories of the flow. 

The 2D incompressible Euler equation on a domain $\Lambda$ is given by
\begin{align}\label{eq:Euler}
\begin{split}
u_t + (u\cdot\grad)u = -\grad P \quad \tx{ in } \Lambda, \\
\grad\cdot u = 0 \quad \tx{ in } \Lambda, \\
u\cdot n = 0 \quad \tx{ on } \partial \Lambda.
\end{split}
\end{align}
Here $u$ is the velocity, $P$ is the pressure and $n$ is the outward unit normal. The vorticity is $\omega = \grad \times u = \partial_{x_1}u_2 - \partial_{x_2}u_1$ and satisfies the transport equation
\begin{align}\label{eq:vorticity}
\omega_t + u\cdot\grad \omega = 0 \quad \tx{ in } \Lambda.
\end{align}
One can recover the velocity from the vorticity by the Biot-Savart law $u = \grad^\perp \Delta^{-1} \w$ where $\Delta$ is the Dirichlet Laplacian and $\grad^\perp  =  (-\partial_{x_2}, \partial_{x_1})$. The 2D Euler equation has several conserved quantities, chief among them being $\norm[L^p(\Lambda)]{\w(\cdot,t)}$ for any $1\leq p\leq \infty$. This is used in an essential way to prove any kind of global well-posedness result.

The study of the well-posedness problem for the 2D Euler equation has a long history. There are two important considerations to keep in mind while talking about the well-posedness problem: one is the regularity of the initial vorticity and the other is the regularity of the boundary. Let us first consider the case of both the vorticity and boundary being regular enough. Global well-posedness for strong solutions in smooth domains was proved by Wolibner \cite{Wo33} and H\"older \cite{Ho33} (see also \cite{Mc67,Ki83}). One of the most important works in the well-posedness theory is the work of Yudovich \cite{Yu63} who established global well-posedness for weak solutions on smooth domains for initial data $\w_0 \in \Lone(\Lambda)\cap\Linfty(\Lambda)$ (see also \cite{Ba72,Te75}). The uniqueness result of Yudovich used the Eulerian formulation and relied on the Calder\'on Zygmund inequalities
\begin{align}\label{eq:CZineq}
\norm[L^p(\Lambda)]{\grad u(\cdot,t)} \leq Cp\norm[L^p(\Lambda)]{\w(\cdot,t)} \qq \tx{ for all } p\in [2,\infty).
\end{align}
Later on Marchioro and Pulvirenti \cite{MaPu94} gave a different proof of uniqueness by using the Lagrangian formulation which relied on the log-Lipschitz nature of the velocity
\begin{align}\label{eq:logLipest}
\sup_{x,y \in \Lambda} \frac{\abs{u(x,t) - u(y,t)}}{\abs{x-y}\max\cbrac{-\ln\abs{x-y}, 1}} \leq C\norm[\Lone(\Lambda)\cap\Linfty(\Lambda)]{\w(\cdot,t)}.
\end{align}
These estimates hold for $C^{1,1}$ domains but may not hold for less regular domains (see \cite{JeKe95}). 

For the case of initial vorticity being less regular, global existence of weak solutions was proved by DiPerna and Majda \cite{DiMa87} for $\w_0 \in \Lone(\Rsp^2)\cap L^p(\Rsp^2)$ for $p>1$ and by Delort \cite{De91} for $\w_0 \in H^{-1}(\Rsp^2)\cap \mathcal{M}_+(\Rsp^2)$ (here $\mathcal{M}_+$ is the space of positive Radon measures). Uniqueness is not expected in general in this case and this is a major open problem (see the works \cite{Vi18a,Vi18b, BrSh21,BrMu20}).

For the case of boundary being less regular, global existence of weak solutions for bounded convex domains was proved by Taylor \cite{Ta00} and for arbitrary simply connected bounded domains (and exterior domains) was proved by Gerard-Varet and Lacave \cite{GeLa13,GeLa15}. Both results prove existence for initial vorticity $\w_0 \in \Lone(\Lambda)\cap L^p(\Lambda)$ or $\w_0 \in H^{-1}(\Lambda)\cap \mathcal{M}_+(\Lambda)$. However even for $\w_0 \in \Lone(\Lambda)\cap \Linfty(\Lambda)$ the question of uniqueness is a major open problem. It is important to note that if the domain is less regular, then the uniqueness question does not become simpler even if the initial vorticity is assumed to be smooth, as the regularity of the vorticity can be destroyed at a later time (see \cite{KiZl15, AlCrMa19}).

There have been some recent works that establish uniqueness for rough domains with initial vorticity $\w_0 \in \Lone(\Lambda)\cap\Linfty(\Lambda)$. One strategy used was to identify domains rougher than $C^{1,1}$ which satisfy either \eqref{eq:CZineq} or \eqref{eq:logLipest} and use this to prove uniqueness. This was first achieved by Bardos, Di Plinio and Temam \cite{BaDiTe13} for rectangle domains and for $C^2$ domains which allow corners of angle $\pi/m$ for $m\in \Nsp, m\geq 2$. Later Lacave, Miot and Wang \cite{LaMiWa14} proved uniqueness for $C^{2,\alpha}$ domains with a finite number of acute angled corners, and then Di Plinio and Temam \cite{DiTe15} proved uniqueness for $C^{1,1}$ domains with finitely many acute angled corners. Note that for angles bigger than $\pi/2$, the estimates \eqref{eq:CZineq} and \eqref{eq:logLipest} fail to hold and uniqueness is open in general. 
Another strategy used to prove uniqueness is to prove it for initial vorticity which is constant around the boundary. The idea behind this strategy is that if the vorticity is constant around the boundary, then the uniqueness proof of \cite{MaPu94} works, as in this case one only needs the estimate \eqref{eq:logLipest} for $x,y \in K$ where $K\subset \Lambda$ is a compact set outside of which the vorticity is constant. The strategy thus reduces to showing that if the vorticity is initially constant around the boundary, then it remains constant for later times. Lacave \cite{La15} proved that if the domain is $C^{1,1}$ with finitely many corners with angles greater than $\pi/2$ and $\w_0$ is constant around the boundary and has a definite sign, then $\w$ remains constant around the boundary for all time and the weak solutions are unique.  Lacave and Zlato\v{s} \cite{LaZl19}  proved the same result removing the restriction of definite sign on $\w_0$ but keeping the initial vorticity constant around the boundary and the corners are now only allowed to be in $(0,\pi)$.  Recently Han and Zlato\v{s} \cite{HaZl21} generalized the results of \cite{La15,LaZl19}, by proving uniqueness in more general domains which include convex domains, but for initial vorticity which is still constant around the boundary.

In this paper we consider the uniqueness question for a domain which does not satisfy \eqref{eq:CZineq} or \eqref{eq:logLipest} and has non-constant initial vorticity around the boundary. In this case, the methods used previously to prove uniqueness cannot work and new ideas are needed. To state our result fix $\half<\nu<1$ and let $\Omega $ and $\Omegaplus$ be the domains 
\begin{align*}
\Omega = \cbrac{re^{i\theta} \in \Csp \suchthat r>0 \tx{ and } 0<\theta<\nu\pi} \qq \Omegaplus = \cbrac{re^{i\theta} \in \Csp \suchthat r>0 \tx{ and } 0<\theta<\frac{\nu\pi}{2}}.
\end{align*}
Let the initial vorticity be $\w_0 := \w(\cdot,0)$.  We assume that the initial vorticity satisfies 
\begin{align}\label{assump}
\w_0 \in \Lone(\Omega) \cap \Linfty(\Omega) \tx{ along with } \supp\brac{\w_0} \subset \Omegaplusbar \tx{ and } \w_0 \geq 0.
\end{align}
We can now state our main result.
\begin{thm}\label{thm:main}
Consider the Euler equation in $\Omega$ with initial vorticity $\w_0$ satisfying \eqref{assump}. Then there exists a unique Yudovich weak solution in the time interval $[0,\infty)$ with this initial data. 
\end{thm}

See \secref{sec:weak} for a precise definition of Yudovich weak solutions. This is the first result which proves uniqueness when the domain does not satisfy the estimates \eqref{eq:CZineq} or \eqref{eq:logLipest} and when the initial vorticity is non-constant around the corner. As we have non-constant vorticity around the corner, we have to use fundamentally new ideas to prove uniqueness and we explain the new approach below. 

The assumptions on the vorticity in the above theorem can be slightly relaxed. First instead of the assumption that $ \supp\brac{\w_0} \subset \Omegaplusbar$, we only need the assumption that there exists a neighborhood $U$ of $\partial \Omega$ such that $ \supp\brac{\w_0}\cap U \subset \Omegaplusbar$. The proof in this paper goes identically for this case. Furthermore by modifying the proof of \propref{prop:flow} one can establish uniqueness if $ \supp(\w_0) \subset \cbrac{re^{i\theta} \in \Csp \suchthat r\geq 0 \tx{ and } 0 \leq \theta \leq \beta(\nu)\nu\pi}$ for some $\beta(\nu)>\half$. Similarly the assumption of $\w_0 \geq 0$ can be relaxed slightly to include  negative vorticity in some places, by ensuring that a version of \propref{prop:flow} and of \lemref{lem:postime} are still satisfied. Moreover if one only cares about short time uniqueness, then the assumption of $\w_0 \geq 0$ in the theorem can be replaced by the condition $b_0 >0$, where $b_0$ is defined in \eqref{eq:bzero}. This is because the only place we really need the condition $\w_0 \geq 0$ is \lemref{lem:postime} which is useful only for long time uniqueness. 

For short time uniqueness one can also handle the case of multiple obtuse corners, if the vorticity is non-negative (or more generally if $\btil(\cdot, 0)$ is positive at each corner, see \eqref{eq:btilnew}, \eqref{eq:bzero}) and one imposes the condition that the support of the vorticity in a small ball around each obtuse corner lies on one side of the angle bisector. However global in time uniqueness in such a case would not follow in general from the arguments in this paper, as the vorticity around one obtuse corner could touch another obtuse corner in finite time. Similarly the arguments used in this paper do not imply global in time uniqueness for a bounded domain with one corner as the vorticity leaving the corner will come back to the corner in finite time.

To understand the difficulty of the problem, consider the particle trajectories in the support of the vorticity. As the velocity near the corner is very far from Lipschitz, the particles near the corner move in a manner similar to the ODE \eqref{eq:sampleODE}. It is well known that if the vorticity is non-negative then the particle trajectories near the boundary move to the right (see for example \cite{KiZl15, ItMiYo16}), however what may happen is that there could potentially be two different solutions of the Euler equation with both their Lagrangian trajectories moving to the right but at a different rate. In fact the main enemy in proving uniqueness is when the particle trajectories for the two solutions are very close to each other but are not the same. As the velocity is far from Lipschitz near the corner and the transport equation (for the vorticity) is nonlinear, all current tools available to control the distance between the trajectories are inadequate to prove uniqueness.  To the best of our knowledge the method we employ to overcome this difficulty is completely novel.

Let us now explain the main idea of the proof. Let $x(t) \in \Omegabar$ be the position of the particle which starts at the corner i.e. $x(0) = 0$. From \eqref{eq:Xep}, \eqref{eq:bep} and \eqref{eq:bzero} we see that heuristically
\begin{align}\label{eq:sampleODE}
\frac{\diff x}{\diff t} = x^{\frac{1}{\nu} -1 } \qq x(0) = 0.
\end{align}
Observe that as $\half<\nu<1$, the function $x^{\frac{1}{\nu} -1 }$ is not Lipschitz and hence one cannot use the Picard-Lindel\"of theorem to prove uniqueness of this ODE. Indeed one sees that this ODE has several solutions. However as the vorticity is non-negative, the flow automatically chooses the solution with the property that the particle moves to the right i.e. $x(t)>0$ for $t>0$. With this constraint the ODE has a unique solution, namely $x(t) = \sqbrac{\brac{\frac{2\nu-1}{\nu}}t}^{\frac{\nu}{2\nu-1}}$. 

So essentially any method employed to prove \thmref{thm:main} has to be strong enough that it can prove the uniqueness of the above ODE (with the constraint $x(t)>0$ for $t>0$). Hence the idea is to find a good method to prove the uniqueness of the above ODE problem and generalize it to prove uniqueness for the Euler equation. There are several ways to prove uniqueness of this ODE problem such as by directly comparing two solutions or by using a change of variable. As generalizing a comparison argument to the Euler equation looks difficult, let us see  how a change of variables argument can be used to show uniqueness for the ODE problem. Consider two solutions $x_1(t)$ and $x_2(t)$ of \eqref{eq:sampleODE} with $x_i(t)>0$ for $t>0$ for $i=1,2$. Letting $E(t) = \abs*[\big]{ x_1(t)^{\brac{2 - \frac{1}{\nu}}} -  x_2(t)^{\brac{2 - \frac{1}{\nu}}}}$ we see that
\begin{align*}
\frac{\diff E}{\diff t} \leq \abs{\brac{2 - \frac{1}{\nu}}x_1^{1 - \frac{1}{\nu}}\frac{\diff x_1}{\diff t} - \brac{2 - \frac{1}{\nu}}x_2^{1 - \frac{1}{\nu}}\frac{\diff x_2}{\diff t} } \leq 0.
\end{align*}
Therefore $E(t) = 0$ for all time and hence we have uniqueness of the ODE problem. Although this approach looks promising, we ran into several technical issues with generalizing this method to the Euler equation with the main issue being the low regularity of the vorticity (which cannot be overcome by simply assuming smoother initial vorticity). To overcome this problem, observe that using \lemref{lem:powers} we can write the energy as 
\begin{align*}
E(t) =  \abs*[\big]{ x_1(t)^{\brac{2 - \frac{1}{\nu}}} -  x_2(t)^{\brac{2 - \frac{1}{\nu}}}} \approx \abs{x_1(t) - x_2(t)}\min\cbrac{x_1(t)^{\brac{1 - \frac{1}{\nu}}}, x_2(t)^{\brac{1 - \frac{1}{\nu}}}}.
\end{align*}
Now as the unique solution to the ODE with the constraint $x(t)>0$ for $t>0$ is given by $x(t) = \sqbrac{\brac{\frac{2\nu-1}{\nu}}t}^{\frac{\nu}{2\nu-1}}$, we see that
\begin{align*}
E(t) \approx  \brac{t^{\frac{\nu}{2\nu - 1}}}^{\brac{1 - \frac{1}{\nu}}} \abs{x_1(t) - x_2(t)} =  t^{- \brac{\frac{1-\nu}{2\nu - 1}}} \abs{x_1(t) - x_2(t)}.
\end{align*}
Hence we can think of the energy as having a time depending weight. We managed to generalize this idea of using a time dependent weight in the energy and this is the way we prove \thmref{thm:main}. Let us illustrate our method by providing a proof of the ODE problem using this method.

\smallskip
\emph{Step 1}: Consider two solutions $x_1(t)$ and $x_2(t)$ of \eqref{eq:sampleODE} with $x_i(t)>0$ for $t>0$ for $i=1,2$. Prove that given $0< \ep < 1$ there exists $T>0$ so that for $i=1,2$ we have $x_i(t) \geq \sqbrac{\brac{\frac{2\nu-1}{\nu}}(1-\ep)t}^{\frac{\nu}{2\nu-1}}$ for $t \in [0,T]$. 

This can be proven by observing that $\frac{\diff x_i}{\diff t} \geq (1-\ep)x^{\frac{1}{\nu} -1 }$. Integrating this inequality we get the required estimate. 

\smallskip
\emph{Step 2}: Consider the energy $E_1(t) = \abs{x_1(t) - x_2(t)}$. By a simple inequality we prove in \lemref{lem:powers} part (1), we obtain
\begin{align*}
\frac{\diff E_1}{\diff t} \leq \abs*[\Big]{x_1^{\frac{1}{\nu} -1} - x_2^{\frac{1}{\nu} - 1}} \lesssim_\nu \abs{x_1 - x_2}^{\frac{1}{\nu} -1} = E_1^{\frac{1}{\nu} -1} \qq \tx{ and } E_1(0) = 0.
\end{align*}
Hence by integration we get $E_1(t) \lesssim_\nu t^{\frac{\nu}{2\nu-1}}$ in the time interval $[0,T]$.

\smallskip
\emph{Step 3}: Consider the energy $E(t) = t^{-\alpha}E_1(t) = t^{-\alpha} \abs{x_1(t) - x_2(t)}$.  Observe that if $0<\alpha<\frac{\nu}{2\nu-1}$ then by step 2, $E(t) \to 0$ as $t \to 0^+$. Now
\begin{align*}
\frac{\diff E}{\diff t} & \leq \brac{\frac{-\alpha}{t}}t^{-\alpha}\abs{x_1(t) - x_2(t)} + t^{-\alpha}\abs*[\Big]{x_1^{\frac{1}{\nu} -1} - x_2^{\frac{1}{\nu} - 1}} \\
& = \cbrac*[\bigg]{ \brac{\frac{-\alpha}{t}} + \abs*[\bigg]{\frac{x_1^{\frac{1}{\nu} -1} - x_2^{\frac{1}{\nu} - 1}}{x_1 - x_2}}}t^{-\alpha}\abs{x_1(t) - x_2(t)} \\
& \leq \cbrac*[\bigg]{ \brac{\frac{-\alpha}{t}} + \brac{\frac{1}{\nu} - 1}\max\cbrac*[\Big]{x_1^{\frac{1}{\nu} - 2}, x_2^{\frac{1}{\nu} - 2} }}t^{-\alpha}\abs{x_1(t) - x_2(t)}.
\end{align*}
Now using step 1, we get 
\begin{align*}
\frac{\diff E}{\diff t} & \leq \cbrac*[\bigg]{ \brac{\frac{-\alpha}{t}} + \brac{\frac{1}{\nu} - 1}\sqbrac{\brac{\frac{2\nu-1}{\nu}}(1-\ep)t}^{-1} }t^{-\alpha}\abs{x_1(t) - x_2(t)} \\
& = \cbrac{ - \alpha + \frac{1-\nu}{(2\nu -1)(1-\ep)} }t^{-\alpha-1}\abs{x_1(t) - x_2(t)}.
\end{align*}
As $\nu>1-\nu$, we see that we can suitably choose $\alpha$ and $\ep$ so that $\frac{\diff E}{\diff t}  \leq 0$. Hence $E(t) =0$ in $[0,T]$ and this proves $x_1(t) = x_2(t)$ for $t\in [0,T]$.

\smallskip
\emph{Step 4}: Uniqueness for $t\geq T$ follows from the Picard-Lindel\"of theorem by observing that there exists a $c>0$ such that $x_1(t), x_2(t) \geq c$ for all $t\geq T$. 

\medskip

The proof of \thmref{thm:main} closely follows the above strategy. The analogs of step 1-4 are \propref{prop:flow}, \propref{prop:Eone}, proof of main \thmref{thm:main} and \lemref{lem:postime} respectively. The analog of the Picard-Lindel\"of theorem is the uniqueness proof given in Sec 2.3 of \cite{MaPu94}. The assumptions on the initial vorticity \eqref{assump} are imposed so that the uniqueness problem for the Euler equation behaves in a similar manner to uniqueness problem of the ODE \eqref{eq:sampleODE}. The assumption $\w_0 \geq 0$ ensures that the particles near the boundary always move to the right and we prove this in \propref{prop:flow} and \lemref{lem:postime}. The assumption of $ \supp\brac{\w_0} \subset \Omegaplusbar$ ensures that essentially the particles in the support of the vorticity move away from the corner and this is shown in \propref{prop:flow}. Both of these properties are the analogs of the constraint $x(t)>0$ for $t>0$ for the ODE problem solved above. As mentioned before, the restrictions on the vorticity can be slightly relaxed.  In addition to the uniqueness result, we also prove the existence of weak solutions as previous existence results do not exactly cover our situation and hence we include the proof for the sake of completeness.

The paper is organized as follows: In \secref{sec:notation} we introduce the notation and derive the flow equation. In \secref{sec:weak} we prove the existence of weak solutions and establish properties of the flow map and in particular prove that the flow near the boundary moves to the right and that the support of the vorticity moves away from the corner. Finally in \secref{sec:energy} we prove the energy estimates required to prove \thmref{thm:main}. The appendix \secref{sec:appendix} contains some basic estimates that we use throughout the paper.

\medskip
\noindent \textbf{Acknowledgment}: The authors thank the anonymous reviewer for the suggestions to improve the clarity of the exposition.

\section{Notation and Preliminaries}\label{sec:notation}

Let $\Hsp = \cbrac{(x_1,x_2) \in \Csp \suchthat x_2>0}$ denote the upper half plane and we will identify $\Rsp^2 \simeq \Csp$. Let $\Hspplus = \cbrac{(x_1,x_2) \in \Csp \suchthat x_1>0 \tx{ and } x_2>0}$ and denote a ball of radius $r$ by $B_r(z_0) = B(z_0, r) = \cbrac{z\in \Csp \suchthat \abs{z-\z_0} <r}$. Let $\Dsp = B_1(0)$ be the unit disc and let $S^1 = \partial \Dsp$. For $f \in \Lone\cap\Linfty$ we write $\norm[\Lone\cap\Linfty]{f} = \norm[1]{f} + \norm[\infty]{f}$. For $z_1, z_2 \in \Hsp$, let $[z_1, z_2]$ denote the line segment connecting $z_1$ and $z_2$. We define the function $\phi:[0,\infty) \to \Rsp$ as $\phi(0) =0 $ and for $x>0$ as  
\begin{align}\label{def:phi}
\phi(x) = x\max\cbrac{-\ln(x), 1}.
\end{align}
Observe that $\phi$ is a continuous increasing function on $[0,\infty)$ with  $x\leq \phi(x) $ for all $x\geq 0$ and that $\phi$ is a concave function on the interval $[0,1/10]$. Also observe that if $c\geq 1$ then $\phi(cx) \leq c\phi(x)$ for all $x\geq 0$. 

We now introduce a notation for certain integrals which appear in our computations. Let $f\in \Linfty(\Csp)$ and let $z_1, \cdots, z_n \in \Csp$ be $n$ distinct complex numbers. If $ \alpha_1,\cdots, \alpha_n \geq 0$ and $0\leq r, R \leq \infty$ we define
\begin{align}\label{def:I}
I((z_1,\alpha_1), \cdots, (z_n,\alpha_n):(f, r,R)) = \int_{A} \frac{1}{\abs{s-z_1}^{\alpha_1}\cdots\abs{s-z_n}^{\alpha_n}}\abs{f(s)} \diff s,
\end{align} 
where $A = \nobrac{B(z_1,r)^c\cap\cdots\cap B(z_n,r)^c}\cap\nobrac{ B(z_1,R)\cap\cdots\cap B(z_n,R)}$. Observe that the set $A$ is the set of all $s\in \Csp$ with distance to the set $\cbrac{z_1,\cdots,z_n}$ between $r$ and $R$.

We write $a \lesssim b$ if there exists a universal constant $C>0$ so that $a\leq Cb$. We write $a\lesssim_{\eta} b$ if there exists a constant $C = C(\eta)>0$ depending only on $\eta$ so that $a\leq Cb$. Similar definitions for $\lesssim_{\eta_1,\eta_2}$, $\lesssim_{\eta_1,\eta_2, \eta_3}$ etc. We write $a \approx b$ if $a \lesssim b$ and $b\lesssim a$. Similarly we write $a \approx_\eta b$ if $a\lesssim_\eta b$ and $b \lesssim_\eta a$ etc. In this paper we fix the angle of the domain $\Omega$ as $\nu \pi$ (with $1/2< \nu < 1$) and we will suppress the dependence of constants on $\nu$ as it shows up quite frequently. 

Let us now derive the equation of the flow. As we are only interested in the flow in $\Omega$ and domains which smoothly approximate $\Omega$, we derive the equation only for such domains. Let $\Lambda$ be a domain homeomorphic to $\Hsp$ or $\Dsp$ with $\partial\Lambda$ being correspondingly homeomorphic to $\Rsp$ or $S^1$. If the Green's function of the domain $\Lambda$ is $G_{\Lambda}(x,y)$, then the kernel of the Biot-Savart law is $ K_{\Lambda}(x,y) := \grad_x^\perp G_{\Lambda}(x,y)$ with $\grad^\perp_x  =  (-\partial_{x_2}, \partial_{x_1})$. Let $\Psi:\Lambda \to \Hsp$ be a Riemann map and observe that $\Psi$ extends continuously to $\Lambdabar$ by Carath\'{e}odary's theorem. Fix $\ztwo \in \Lambda$ and let $f: \Lambda\backslash\cbrac{\ztwo} \to \Csp$ be defined as  
\begin{align}\label{eq:f}
f(\z) =  \frac{1}{2\pi}\ln\brac*[\Bigg]{\frac{\Psi(z) - \Psi(\ztwo) }{\Psi(z) - \Psibar(\ztwo) }}.
\end{align}
Clearly $f$ is holomorphic and we have for $\zone,\ztwo \in \Lambda$, $\zone\neq \ztwo$
\begin{align*}
G_{\Lambda}(\zone,\ztwo) =  \frac{1}{2\pi}\ln\abs*[\Bigg]{\frac{\Psi(\zone) - \Psi(\ztwo) }{\Psi(\zone) - \Psibar(\ztwo) }} = \Real\cbrac{f(\zone)}.
\end{align*}
Hence
\begin{align*}
K_{\Lambda}(\zone,\ztwo) & = \Real\begin{pmatrix} -\partial_{x_2} f(\zone) \\ \partial_{x_1} f(\zone) \end{pmatrix} \\
& = \Real(-i f_{x_1}(\zone)) + i \Real(f_{x_1}(z_1)) \\
& = i\overline{f_z}(z_1).
\end{align*}
Then from \eqref{eq:f} we have
\begin{align}\label{eq:Kernel}
K_{\Lambda}(\zone,\ztwo) =  \brac{\frac{i}{2\pi}}\Psizbar(\zone)\sqbrac{\frac{1}{\Psibar(\zone) - \Psibar(\ztwo)} - \frac{1}{\Psibar(\zone) - \Psi(\ztwo)}}.
\end{align}
If $\w(\cdot,t)$ is the vorticity at time $t$, then from the Biot-Savart law we see that 
\begin{align*}
u(z_1,t) = \int_{\Lambda} K_{\Lambda}(z_1,z_2)\w(z_2,t)\diff z_2.
\end{align*}
Now the equation for the flow $X:\Lambda\times [0,\infty) \to \Lambda$ is given by
\begin{align*}
\frac{\diff X(x,t)}{\diff t} & = u(X(x,t),t) = \int_{\Lambda} K_{\Lambda}(X(x,t),z)\w(z,t) \diff z.
\end{align*}
Hence we have
\begin{align*}
\frac{\diff X(x,t)}{\diff t} = \brac*[\Big]{\frac{i}{2\pi}}\Psizbar(X(x,t)) \int_{\Lambda}\sqbrac{\frac{1}{\Psibar(X(x,t)) - \Psibar(z)} - \frac{1}{\Psibar(X(x,t)) - \Psi(z)}}\w(z,t) \diff z.
\end{align*}
Define the function $b:\Lambda\times [0,\infty) \to \Csp$ as 
\begin{align}\label{eq:b}
b(x,t) =  \brac*[\Big]{\frac{i}{2\pi}} \int_{\Lambda}\sqbrac{\frac{1}{\Psibar(x) - \Psibar(z)} - \frac{1}{\Psibar(x) - \Psi(z)}}\w(z,t) \diff z.
\end{align}
Hence the equation for $X$ can be written as
\begin{align}\label{eq:X}
\frac{\diff X(x,t)}{\diff t} = b(X(x,t),t)\Psizbar(X(x,t)).
\end{align}

We now convert the flow equation above in $\Lambda$ to a flow equation on $\Hsp$. For $x\in \Lambda$, let $y \in \Hsp$ be given by $y = \Psi(x)$. Consider the flow $Y:\Hsp\times[0,\infty) \to \Hsp $ given by 
\begin{align}\label{def:Y}
Y(y,t) = \Psi(X(x,t))
\end{align}
and define $\btil:\Hsp \times [0,\infty) \to \Csp$ as $\btil(y,t) = b(x,t)$. Then $\btil(Y(y,t),t) = b(X(x,t),t)$ and we have
\begin{align}\label{eq:Y}
\frac{\diff Y(y,t)}{\diff t} = \btil(Y(y,t),t)\abs{\Psi_z \compose \Psi^{-1} (Y(y,t))}^2.
\end{align}
We can write a simple formula for $\btil$. As $y = \Psi(x)$ we see that
\begin{align*}
\btil(y,t) = b(x,t) & =  \brac*[\Big]{\frac{i}{2\pi}} \int_{\Lambda}\sqbrac{\frac{1}{\Psibar(x) - \Psibar(z)} - \frac{1}{\Psibar(x) - \Psi(z)}}\w(z,t) \diff z \\
& =  \brac*[\Big]{\frac{i}{2\pi}} \int_{\Lambda}\sqbrac{\frac{1}{\ybar - \Psibar(z)} - \frac{1}{\ybar - \Psi(z)}}\w(z,t) \diff z.
\end{align*}
Next, we change variables by setting $s = \Psi(z)$ with $s\in\Hsp$ and observe that $\diff s = \abs{\Psiz(z)}^2\diff z$ and hence $\diff z = \abs{\Psiz\compose \Psi^{-1} (s)}^{-2} \diff s$. Defining $\wtil:\Hsp\times[0,\infty) \to \Rsp$ as $\wtil(s,t) = \w(z,t)$ we get
\begin{align}\label{eq:btil}
\btil(y,t) =  \brac*[\Big]{\frac{i}{2\pi}} \int_{\Hsp}\sqbrac{\frac{1}{\ybar - \sbar} - \frac{1}{\ybar - s}}\wtil(s,t)  \abs{\Psiz\compose \Psi^{-1} (s)}^{-2} \diff s.
\end{align}

\section{Weak Solutions}\label{sec:weak}

We now give the definition of Yudovich weak solutions and prove their existence for the domain $\Omega$. We prove the existence of weak solutions in $\Omega$ as the previous existence results do not apply directly. The existence proof of Taylor \cite{Ta00} and Gerard-Varet and Lacave \cite{GeLa13,GeLa15} are either for bounded domains or for exterior domains. We modify the existence proof for $\Rsp^2$ as given in the book by Majda and Bertozzi \cite{MaBe02} to prove existence of weak solutions in $\Omega$. Similar to  \cite{Ta00} and \cite{GeLa13}, we approximate the domain $\Omega$ by smooth domains $\Omega_\ep$ and then take a limit as $\ep \to 0$. Even though the method for proving existence of weak solutions is quite standard, we include it for the sake of completeness. 

For the definition of weak solution we closely follow the definition as given in \cite{GeLa13,GeLa15}. If $\Lambda$ is homeomorphic to $\Dsp$ with $\partial\Lambda$ homeomorphic to $S^1$, we use the definition of Yudovich weak solution as given in \cite{LaZl19}. Hence now consider a domain $\Lambda$ homeomorphic to $\Hsp$ with $\partial\Lambda$ being homeomorphic to $\Rsp$. We are mostly interested in the case with $\Lambda = \Omega$ and the definition below is tailored to domains similar to $\Omega$. For more general domains a slightly different definition as compared to the one below may be needed. We say that $(u,\w)$ is in the Yudovich class in the time $[0,T)$ if 

\begin{align}\label{eq:Yudovich}
\begin{split}
& u \in \Linfty_{loc}([0,T); \Ltwo_{loc}(\Lambdabar)), \qq  \w = \grad \times u \in \Linfty([0,T); \Lone(\Lambda) \cap \Linfty(\Lambda)), \quad\\
\tx{ and}\enspace &  u(\cdot,t) \in C(\Lambdabar) \enspace  \text{ with }  \lim_{R\to \infty} \sup_{\abs{x} \geq R} \abs{u(x,t)} = 0 \quad \tx{ for a.e. } t\in [0,T).
\end{split}
\end{align}
Now let
\begin{align*}
G_c(\Lambda) = \cbrac{ h \in \Ltwo_c(\Lambdabar) \suchthat h = \grad p \tx{ for some } p \in H^1_{loc}(\Lambda)}.
\end{align*}
Consider initial data $(u_0,\w_0)$ satisfying
\begin{align}\label{eq:initial}
\begin{split}
&u_0 \in C(\Lambdabar), \qq \w_0 = \grad \times u_0 \in \Lone(\Lambda)\cap\Linfty(\Lambda), \quad  \quad  \\
 &  \lim_{R \to \infty} \sup_{\abs{x} \geq R} \abs{u_0(x)} = 0 \qquad  \tx{ and } \quad \int_{\Lambda} u_0 \cdot h = 0 \quad \forall h \in G_c(\Lambda).
 \end{split}
\end{align}
\begin{definition}
We say that $(u,\w)$ is a Yudovich weak solution to the Euler equation \eqref{eq:Euler} with initial condition $(u_0,\w_0)$ in the time interval $[0,T)$, if $(u,\w)$ is in the Yudovich class \eqref{eq:Yudovich} and satisfies
\begin{align}\label{eq:weak}
\int_0^T \int_{\Lambda} \w(\pt \varphi + u\cdot \grad \varphi) \diff x \diff t = - \int_{\Lambda} \w_0 \varphi(\cdot,0) \diff x \qq \forall \varphi \in C^{\infty}_c (\Lambda\times[0,T)),
\end{align}
and for $a.e.$ $t\in [0,T)$ we have  
\begin{align}\label{eq:divweak}
\int_{\Lambda} u(\cdot,t)\cdot h = 0 \qq \forall h\in G_c(\Lambda).
\end{align}
\end{definition}

Note that we have given the definition of Yudovich weak solutions as weak solutions to the transport equation. It can be shown that this is equivalent to the definition of weak solution to the Euler equation, see Remark 1.2 of \cite{GeLa15}. Now for $\half < \nu < 1$ let us now consider the domain $\Omega = \cbrac{re^{i\theta} \in \Csp \suchthat r>0 \tx{ and } 0<\theta<\nu\pi} $ and the Riemann map $\Psi : \Omega \to \Hsp$ given by $\Psi(z) = z^{\frac{1}{\nu}}$. We want to prove the existence of Yudovich weak solutions for this domain and understand the properties of the flow map.

\subsection{Existence of weak solutions}

In this section we prove the existence of Yudovich weak solutions in $\Omega$. We will approximate $\Omega$ with smooth bounded domains $\Omega_\ep$. For $\ep = 0$ we define $\Omega_0 := \Omega$ and for $0<\ep\leq 1$ we define $\Omega_\ep$ as
\begin{align*}
\Omega_\ep := \cbrac{\z^\nu \suchthat z \in B\brac{i\brac*[\Big]{\ep + \frac{1}{2\ep}}, \frac{1}{2\ep}}}.
\end{align*} 
It is easy to see that for $0<\ep\leq 1$,  $\Omega_\ep$ are smooth bounded domains with $\Omega_{\ep_1} \subset \Omega_{\ep_2}$ for $\ep_1 \geq \ep_2$ and $\cup_{0<\ep\leq 1} \Omega_\ep = \Omega$. Let $\Psi_\ep : \Omega_\ep \to \Hsp$ be a Riemann map and let $\Psi_\ep^{-1}$ be its inverse. We now give an explicit formula for these maps. First consider the map
\begin{align}\label{eq:mapbasic}
z \mapsto \frac{z}{1 - i\ep z} + i\ep \qq \tx{ for } z \in \Hsp.
\end{align}
We see that for this map $0 \mapsto i\ep$, $\frac{i}{\ep} \mapsto i\brac{\ep + \frac{1}{2\ep}}$ and $\infty \mapsto i\brac{\ep + \frac{1}{\ep}}$. Hence this maps the upper half plane $\Hsp$ to the ball $B\brac{i\brac*[\Big]{\ep + \frac{1}{2\ep}}, \frac{1}{2\ep}}$. Let the topmost point of this ball be defined as $c(\ep) := i\brac{\ep + \frac{1}{\ep}}$. Hence $(c(\ep))^\nu \in \partial \Omega_\ep$. Hence we see that 
\begin{align}\label{eq:Psiinvep}
\Psi_\ep^{-1}(z) = \sqbrac{\frac{z}{1 - i\ep z} + i\ep}^\nu = \sqbrac{c(\ep) + \frac{1}{\ep(i + \ep\z)}}^\nu.
\end{align}
Therefore we obtain
\begin{align}\label{eq:Psiep}
\Psi_\ep(z) = \frac{z^{\frac{1}{\nu}} - i\ep}{1 + \ep^2 + i\ep z^{\frac{1}{\nu}}}  = -\frac{i}{\ep} + \frac{1}{\ep^2(z^{\frac{1}{\nu}} - c(\ep))}.
\end{align}
For $\ep = 0$ we define the maps
\begin{align}\label{eq:Psiepzero}
\Psi(z) = z^{\frac{1}{\nu}} \qq \tx{ and } \qq \Psi^{-1}(z) = z^\nu.
\end{align}
It is clear that for fixed $z \in \Omega$ we have $\lim_{\ep \to 0} \Psi_\ep(z) = z^{\frac{1}{\nu}} = \Psi(z)$. Similarly for fixed $z \in \Hsp$ we have $\lim_{\ep \to 0} \Psi_\ep^{-1}(z) = z^\nu = \Psi^{-1}(z)$. Also note tht for $\ep =0$, we have from \eqref{eq:Kernel} that
\begin{align}\label{eq:Kepzero}
K_{\Omega}(\zone,\ztwo) =  \brac{\frac{i}{2\pi\nu}}\zonebar^{\brac{\frac{1}{\nu} - 1}} \sqbrac{\frac{1}{\zonebar^{\frac{1}{\nu}} - \ztwobar^{\frac{1}{\nu}}} - \frac{1}{\zonebar^{\frac{1}{\nu}} - \ztwo^{\frac{1}{\nu}} } }.
\end{align}

Let us now prove some basic properties of these maps.
\begin{lemma}\label{lem:Psiep}
Let $0\leq \ep \leq 1$. Then 
\begin{enumerate}
\item Suppose $\ep>0$. Then for $z_1, z_2 \in \Omega_\ep$, we have
\begin{align*}
\abs{\Psi_\ep(z_1) - \Psi_\ep(z_2)} \approx \frac{1}{\ep^2}\frac{\abs*[\big]{z_1^{\frac{1}{\nu}} - z_2^{\frac{1}{\nu}}}}{\abs*[\big]{z_1^{\frac{1}{\nu}} - c(\ep)} \abs*[\big]{z_2^{\frac{1}{\nu}} - c(\ep)}}.
\end{align*}
\item Suppose $\ep>0$. Then for $z \in \Omega_\ep$ we have
\begin{align*}
\abs{\Psi_\ep(z)} \lesssim \frac{1}{\ep^2\abs*[\big]{z^{\frac{1}{\nu}} - c(\ep)}} \qq \tx{ and } \qq \abs{(\Psi_\ep)_z(z)} \approx \frac{1}{\ep^2} \frac{\abs{z}^{\frac{1}{\nu} - 1}}{\abs*[\big]{z^{\frac{1}{\nu}} - c(\ep)}^2}.
\end{align*}

\item For $z_1, z_2 \in \Omega_\ep$ with $z_1 \neq z_2$,  we have
\begin{align*}
\abs{K_{\Omega_\ep}(z_1, z_2)}  \lesssim  \frac{1}{\abs{z_1 - z_2}}.
\end{align*}
\item For $z \in \Hsp$ we have
\begin{align*}
\abs{(\Psi_\ep)_z (\Psi_\ep^{-1} (z))}^{-2} \lesssim \abs{z + i\ep}^{2\nu - 2}.
\end{align*}
\end{enumerate}
\end{lemma}
\begin{proof}
We prove the estimates sequentially:
\begin{enumerate}
\item This estimate follows directly from \eqref{eq:Psiep}.
\item Observe that for any $z \in \Omega_\ep$, we have that $z^\frac{1}{\nu} \in B\brac{i\brac*[\Big]{\ep + \frac{1}{2\ep}}, \frac{1}{2\ep}}$. Hence for any $z\in \Omega_\ep$ we see that 
\begin{align}\label{eq:basicep}
\ep \abs*[\big]{z^\frac{1}{\nu} - c(\ep)} \leq 1.
\end{align}
From this we see that for any $z\in \Omega_\ep$
\begin{align*}
\abs{\Psi_\ep(z)} \leq \frac{2}{\ep^2\abs*[\big]{z^{\frac{1}{\nu}} - c(\ep)}}.
\end{align*}
Now taking a derivative in \eqref{eq:Psiep} we get
\begin{align}\label{eq:Psizep}
(\Psi_\ep)_z(z) = \frac{-z^{\brac{\frac{1}{\nu} - 1}}}{\nu\ep^2(z^\frac{1}{\nu} - c(\ep))^2}.
\end{align}
Therefore the estimate follows.
\item First note that using \lemref{lem:powers} we get
\begin{align*}
\frac{\abs{z_1}^{\frac{1}{\nu} - 1}}{\abs*[\big]{z_1^{\frac{1}{\nu}} - z_2^{\frac{1}{\nu}}}} \lesssim \frac{\abs{z_1}^{\frac{1}{\nu} - 1}}{\abs{z_1 - z_2}\max\cbrac{\abs{z_1}^{\frac{1}{\nu} - 1}, \abs{z_2}^{\frac{1}{\nu} - 1}} } \lesssim \frac{1}{\abs{z_1 - z_2}}.
\end{align*}
Hence it is enough to prove that
\begin{align*}
\abs{K_{\Omega_\ep}(z_1, z_2)} \lesssim \frac{\abs{z_1}^{\frac{1}{\nu} - 1}}{\abs*[\big]{z_1^{\frac{1}{\nu}} - z_2^{\frac{1}{\nu}}}}.
\end{align*}
For $\ep = 0$ this follows directly from \eqref{eq:Kepzero}. Now let $\ep > 0$. From \eqref{eq:Kernel} we see that
\begin{align}\label{eq:absKep}
\abs{K_{\Omega_\ep}(z_1,z_2)} \approx \frac{\abs{(\Psi_\ep)_z(z_1) }\abs{\Psi_\ep(z_2) - \Psibar_\ep(z_2)} }{\abs{\Psi_\ep(z_1) - \Psi_\ep(z_2)}\cbrac{\abs{\Psi_\ep(z_1) - \Psi_\ep(z_2)} + \abs{\Psi_\ep(z_2) - \Psibar_\ep(z_2)}}}.
\end{align}
\textbf{Case 1:} $\abs*[\big]{z_2^\frac{1}{\nu} - c(\ep)} < 2 \abs*[\big]{z_1^\frac{1}{\nu} - c(\ep)}.$

In this case we see from the first estimate of this lemma that
\begin{align*}
\frac{1}{\abs{\Psi_\ep(z_1) - \Psi_\ep(z_2)}} \lesssim \frac{\ep^2 \abs*[\big]{z_1^{\frac{1}{\nu}} - c(\ep)}^2}{\abs*[\big]{z_1^{\frac{1}{\nu}} - z_2^{\frac{1}{\nu}}}}.
\end{align*}
Hence using \eqref{eq:absKep} and the second estimate of this lemma we obtain
\begin{align*}
\abs{K_{\Omega_\ep}(z_1,z_2)} \lesssim \frac{\abs{(\Psi_\ep)_z(z_1) }}{\abs{\Psi_\ep(z_1) - \Psi_\ep(z_2)}} \lesssim \frac{\abs{z_1}^{\frac{1}{\nu} - 1}}{\abs*[\big]{z_1^{\frac{1}{\nu}} - z_2^{\frac{1}{\nu}}}}.
\end{align*}
\textbf{Case 2:} $\abs*[\big]{z_2^\frac{1}{\nu} - c(\ep)} \geq 2 \abs*[\big]{z_1^\frac{1}{\nu} - c(\ep)}$.

In this case we see that $\abs*[\big]{z_2^\frac{1}{\nu} - c(\ep)} \approx \abs*[\big]{z_1^{\frac{1}{\nu}} - z_2^{\frac{1}{\nu}}}$. Hence from the first estimate of this lemma we have
\begin{align*}
\frac{1}{\abs{\Psi_\ep(z_1) - \Psi_\ep(z_2)}} \lesssim \ep^2 \abs*[\big]{z_1^{\frac{1}{\nu}} - c(\ep)}.
\end{align*}
From this estimate, \eqref{eq:absKep} and the second estimate of this lemma we have
\begin{align*}
\abs{K_{\Omega_\ep}(z_1,z_2)} \lesssim  \frac{\abs{(\Psi_\ep)_z(z_1) }\abs{\Psi_\ep(z_2)} }{\abs{\Psi_\ep(z_1) - \Psi_\ep(z_2)}^2 } \lesssim \frac{\abs{z_1}^{\frac{1}{\nu} - 1}}{\abs*[\big]{z_2^{\frac{1}{\nu}} - c(\ep)}} \lesssim \frac{\abs{z_1}^{\frac{1}{\nu} - 1}}{\abs*[\big]{z_1^{\frac{1}{\nu}} - z_2^{\frac{1}{\nu}}}}.
\end{align*}
\item For $\ep = 0$ we see from \eqref{eq:Psiepzero} that
\begin{align}\label{eq:Psizepzero}
\Psi_z (\Psi^{-1} (z)) = \frac{1}{\nu} z^{1 - \nu}.
\end{align}
Hence we get $\abs{\Psi_z (\Psi^{-1} (z))}^{-2} \lesssim \abs{z}^{2\nu -2}$. Now let $\ep >0$. Using \eqref{eq:Psizep} and \eqref{eq:Psiinvep} we see that
\begin{align*}
(\Psi_\ep)_z (\Psi_\ep^{-1} (z)) = -\frac{1}{\nu} (i + \ep z)^2 \sqbrac{\frac{z}{1 - i\ep z} + i\ep}^{1 - \nu}.
\end{align*}
Now using the fact that for all $z \in \Hsp$ we have $\abs{i + \ep z} \geq 1$, we obtain
\begin{align}\label{eq:temp23}
\abs{(\Psi_\ep)_z (\Psi_\ep^{-1} (z))}^{-2} \approx \frac{\abs{z + i\ep + \ep^2 z}^{2\nu -2}}{\abs{i + \ep z}^{2 + 2\nu}} \lesssim \abs{z + i\ep + \ep^2 z}^{2\nu -2}.
\end{align}
We also note that
\begin{align*}
z + i\ep = \frac{(1 + \ep^2)}{(1 + \ep^2)}(z + i \ep) = \frac{(z + i\ep + \ep^2 z) + i \ep^3}{1 + \ep^2}.
\end{align*}
Hence using the fact that for all $z \in \Hsp$ we have $\abs{z + i\ep + \ep^2 z} \geq \ep$, we obtain
\begin{align*}
\abs{z + i\ep} \leq \abs{z + i\ep + \ep^2 z} + \ep^3 \leq \abs{z + i\ep + \ep^2 z} + \ep \leq 2\abs{z + i\ep + \ep^2 z}.
\end{align*}
Combining this with \eqref{eq:temp23} we get the required estimate
\begin{align*}
\abs{(\Psi_\ep)_z (\Psi_\ep^{-1} (z))}^{-2}  \lesssim \abs{z + i\ep + \ep^2 z}^{2\nu -2} \lesssim \abs{z + i\ep}^{2\nu - 2}.
\end{align*}
\end{enumerate}
\end{proof}



We now prove some basic properties of the velocity on the  domains $\Omega_\ep$ and also show that the velocity has to be given by the Biot-Savart law for the domain $\Omega$. 
\begin{lemma}\label{lem:velocity}
Let $0\leq \ep \leq 1$ and let $g \in \Lone(\Omega_\ep) \cap \Linfty(\Omega_\ep)$. If $v(x) = \int_{\Omega_\ep}K_{\Omega_\ep}(x,y)g(y) \diff y$ and $\phi$ is given by \eqref{def:phi} then 
\begin{enumerate}
\item $\norm[\infty]{v} \lesssim \norm[\infty]{g}^{2/3}\norm[1]{g}^{1/3} + \norm[1]{g} \lesssim \norm[\Lone\cap\Linfty]{g}$.
\item If $\ep = 0$, then for $\zone,\ztwo  \in \Omegabar$ we have
\begin{align*}
& \abs{v(\zone) - v(\ztwo)} \\
& \lesssim \norm[\Lone\cap\Linfty]{g}\abs{\zone - \ztwo}\min\cbrac{\abs{\zone}^{\frac{1}{\nu} -2}, \abs{\ztwo}^{\frac{1}{\nu} -2}} +  \norm[\Lone\cap\Linfty]{g}\phi(\abs{\zone - \ztwo}).
\end{align*}
Hence $v$ is continuous on $\Omegabar$. Now let $K \subset \Omega$ be a compact set and suppose $0<\ep_0\leq 1$ is such that $K \subset \Omega_{\ep_0}$. Then there exists $C_{K,\ep_0}>0$ depending only on $K$, $\nu$ and $\ep_0$ such that for all $0\leq \ep \leq \ep_0$
\begin{align*}
\sup_{\zone,\ztwo \in K} \frac{\abs{v(\zone) - v(\ztwo)}}{\phi(\abs{\zone - \ztwo})} \leq C_{K, \ep_0} \norm[\Lone\cap\Linfty]{g}.
\end{align*}
Therefore the velocity in the interior is log-Lipschitz. 
\item For $\ep = 0$ we have $\lim_{R \to \infty} \sup_{\abs{x} \geq R}\abs{v(x)} = 0$.
\item Let $\ep = 0$ and suppose $f \in C(\Omegabar) $ is such that $\grad \times f = g $ in $\Omega$, $\grad \cdot f = 0 $ in $\Omega$, $f\cdot n = 0 $ on $\partial \Omega$ and we have $\lim_{R \to \infty} \sup_{\abs{x} \geq R}\abs{f(x)} = 0$. Then $f = v$.
\end{enumerate}
\end{lemma}
\begin{proof}
We prove each statement individually.
\begin{enumerate}
\item For $\zone \in \Omega_\ep $ we have from  \lemref{lem:Psiep}
\begin{align}\label{eq:vestint}
\begin{split}
\abs{v(\zone)} &  \lesssim  \int_{\Omega_\ep} \abs{K_{\Omega_\ep}(z_1, z)} \abs{g(z)} \diff z \\
& \lesssim \int_{\Omega_\ep} \frac{1}{\abs{\zone - z}}\abs{g(z)} \diff z \\
& \lesssim \int_{B_1(\zone)\cap\Omega_\ep}  \frac{1}{\abs{\zone - z}}\abs{g(z)} \diff z +  \int_{B_1(\zone)^c\cap\Omega_\ep}  \frac{1}{\abs{\zone - z}}\abs{g(z)} \diff z \\
& \lesssim \norm[L^3(B_1(\zone)\cap\Omega_\ep)]{g} + \norm[1]{g} \\
& \lesssim \norm[\infty]{g}^{2/3}\norm[1]{g}^{1/3} + \norm[1]{g}.
\end{split}
\end{align}
Note that all the above estimates are independent of $\ep$. 

\item Consider first the case for $\ep = 0$.  From \eqref{eq:Kepzero} we see that for $\zone \in \Omega$ we have
\begin{align}\label{eq:formv}
v(\zone)   =  \brac{\frac{i}{2\pi\nu}}\zonebar^{\brac{\frac{1}{\nu} - 1}} \int_{\Omega} \sqbrac{\frac{1}{\zonebar^{\frac{1}{\nu}} - \zbar^{\frac{1}{\nu}}} - \frac{1}{\zonebar^{\frac{1}{\nu}} - \z^{\frac{1}{\nu}}   } }g(z) \diff z.
\end{align}
Now using \eqref{eq:formv} and \lemref{lem:powers} we see that for $\zone,\ztwo \in \Omega$ we have
\begingroup
\allowdisplaybreaks
\begin{align*}
& \abs{v(\zone) - v(\ztwo)} \\
& \lesssim \abs*{\zonebar^{\frac{1}{\nu} -1} -  \ztwobar^{\frac{1}{\nu} -1}}\int_{\Omega} \frac{1}{\abs{\zone - z} \max\cbrac{\abs{\zone}^{\frac{1}{\nu} - 1}, \abs{z}^{\frac{1}{\nu} - 1}} }\abs{g(z)} \diff z \\*
& \quad + \abs{\ztwo}^{\frac{1}{\nu} - 1} \int_{\Omega} \frac{\abs*{\zonebar^{\frac{1}{\nu}} - \ztwobar^{\frac{1}{\nu}}} }{\abs*{\zonebar^{\frac{1}{\nu}} - \zbar^{\frac{1}{\nu}}}\abs*{\ztwobar^{\frac{1}{\nu}} - \zbar^{\frac{1}{\nu}}}}\abs{g(z)} \diff z \\
& \lesssim \abs{\zone - \ztwo} \min\cbrac{\abs{\zone}^{\frac{1}{\nu} - 2}, \abs{\ztwo}^{\frac{1}{\nu} - 2}}\norm[\Lone\cap\Linfty]{g} \\*
& \quad +  \abs{\ztwo}^{\frac{1}{\nu} -1}\int_{\Omega} \frac{ \abs{\zone - \ztwo} \max\cbrac{\abs{\zone}^{\frac{1}{\nu} - 1}, \abs{\ztwo}^{\frac{1}{\nu} - 1}} \abs{g(z)}}{ \abs{\zone - \z} \max\cbrac{\abs{\zone}^{\frac{1}{\nu} - 1}, \abs{\z}^{\frac{1}{\nu} - 1}} \abs{\ztwo - \z} \max\cbrac{\abs{\ztwo}^{\frac{1}{\nu} - 1}, \abs{\z}^{\frac{1}{\nu} - 1}}} \diff z \\
& \lesssim \abs{\zone - \ztwo} \min\cbrac{\abs{\zone}^{\frac{1}{\nu} - 2}, \abs{\ztwo}^{\frac{1}{\nu} - 2}}\norm[\Lone\cap\Linfty]{g}  + \abs{\ztwo}^{\frac{1}{\nu} -1}\int_{\Omega} \frac{ \abs{\zone - \ztwo}  \abs{g(z)}}{ \abs{\zone - \z} \abs{\ztwo - \z}\abs{z}^{\frac{1}{\nu} - 1}  } \diff z.
\end{align*}
\endgroup
Now using the notation from \eqref{def:I} we get
\begin{align*}
& \int_{\Omega} \frac{ \abs{\zone - \ztwo}  \abs{g(z)}}{ \abs{\zone - \z} \abs{\ztwo - \z}\abs{z}^{\frac{1}{\nu} - 1}  } \diff z \\
& =  \abs{z_1 - z_2}I((0,\frac{1}{\nu}-1), (z_1,1), (z_2,1): (g\one_{\Omega},0,\infty)).
\end{align*}
Hence using the first estimate of \lemref{lem:Ithree} we see that
\begin{align*}
& \abs{v(\zone) - v(\ztwo)} \\
& \lesssim  \abs{\zone - \ztwo} \min\cbrac{\abs{\zone}^{\frac{1}{\nu} - 2}, \abs{\ztwo}^{\frac{1}{\nu} - 2}}\norm[\Lone\cap\Linfty]{g} \\
& \quad + \abs{\ztwo}^{\frac{1}{\nu} - 1}\norm[\Lone\cap\Linfty]{g} \min\cbrac{\abs{\zone}^{1 -\frac{1}{\nu}},\abs{\ztwo}^{1 -\frac{1}{\nu} } }\phi(\abs{\zone - \ztwo})\\
& \lesssim \abs{\zone - \ztwo}\min\cbrac{\abs{\zone}^{\frac{1}{\nu} -2}, \abs{\ztwo}^{\frac{1}{\nu} -2}} \norm[\Lone\cap\Linfty]{g} +  \phi(\abs{\zone - \ztwo})\norm[\Lone\cap\Linfty]{g}.
\end{align*}

Now let $K \subset \Omega$ be a compact set with $0<\ep_0\leq 1$ such that $K \subset \Omega_{\ep_0}$. Note that in light of the above estimate, we only need to prove the log-Lipschitz nature of the velocity for $0<\ep\leq \ep_0$. Let $K_1 \subset \Hsp$ be the compact set defined by $K_1 = \cbrac{z^\frac{1}{\nu} \suchthat z \in K}$. As $K \subset \Omega_{\ep_0}$, we see that $K_1 \subset B\brac{i\brac*[\Big]{\ep_0 + \frac{1}{2\ep_0}}, \frac{1}{2\ep_0}}$. Hence for all $0<\ep \leq \ep_0$ and $z \in K$ we have
\begin{align*}
\abs*[\big]{z^{\frac{1}{\nu}} - c(\ep)} \approx_{K, \ep_0} \frac{1}{\ep}.
\end{align*}
Hence from \lemref{lem:Psiep}, \eqref{eq:Psizep} and \lemref{lem:powers} we see that for all $0<\ep \leq \ep_0$ and $z_1, z_2 \in K$ we have
\begin{align*}
\abs{(\Psi_\ep)_z(z_1)} & \approx_{K,\ep_0} 1, \\
\abs{(\Psi_\ep)_z(z_1) - (\Psi_\ep)_z(z_2)} & \lesssim_{K,\ep_0} \abs{z_1 - z_2},  \\
\abs{\Psi_\ep(z_1) - \Psi_\ep(z_2)} & \approx_{K,\ep_0} \abs{z_1 - z_2}.
\end{align*}
Similarly using \lemref{lem:Psiep}, \eqref{eq:basicep} and \lemref{lem:powers} we also see that for all $0<\ep \leq \ep_0$, $z_1 \in K$ and $z \in \Omega_\ep$ we have
\begin{align*}
 \frac{1}{\abs{\Psi_\ep(z_1) - \Psi_\ep(z)}} & \approx \ep^2\frac{\abs*[\big]{z_1^{\frac{1}{\nu}} - c(\ep)} \abs*[\big]{z^{\frac{1}{\nu}} - c(\ep)}}{\abs*[\big]{z_1^{\frac{1}{\nu}} - z^{\frac{1}{\nu}}}} \\
& \lesssim_{K,\ep_0} \frac{1}{\abs{z_1 - z}\max\cbrac{\abs{z_1}^{\frac{1}{\nu} - 1}, \abs{z}^{\frac{1}{\nu} - 1}} }  \\
& \lesssim_{K,\ep_0} \frac{1}{\abs{z_1 - z}}.
\end{align*}
Hence from \eqref{eq:Kernel} we see that for all $0<\ep \leq \ep_0$ and $z_1, z_2 \in K$ and $z \in \Omega_\ep$
\begin{align*}
\abs{K_{\Omega_\ep}(z_1, z) - K_{\Omega_\ep}(z_2, z)}  \lesssim_{K,\ep_0} \frac{\abs{z_1 - z_2}}{\abs{z_1 - z}} + \frac{\abs{z_1 - z_2}}{\abs{z_1 - z}\abs{z_2 - z}} .
\end{align*}
Hence using \lemref{lem:Itwo} we see that
\begin{align*}
\abs{v(z_1) - v(z_2)} & \leq \int_{\Omega_\ep} \abs{K_{\Omega_\ep}(z_1, z) - K_{\Omega_\ep}(z_1, z)} \abs{g(z)} \diff \z \\
&  \lesssim_{K,\ep_0} \abs{z_1 - z_2}\norm[\Lone\cap\Linfty]{g} + \phi(\abs{z_1 - z_2})\norm[\Lone\cap\Linfty]{g}.
\end{align*}

\item Let $r>1$ and let $g_1 = g \one_{\cbrac{\abs{x} \leq r}}$ and $g_2 = g - g_1$. Let $v_1(x) = \int_{\Omega}K_{\Omega}(x,y)g_1(y) \diff y$ and $v_2(x) = \int_{\Omega}K_{\Omega}(x,y)g_2(y) \diff y$. For $\zone \in \Omega$ with $\abs{\zone} \geq 2r$ we see from \eqref{eq:vestint} that
\begin{align*}
\abs{v_1(\zone)} \lesssim  \int_{\Omega} \frac{1}{\abs{\zone - z}}\abs{g_1(z)} \diff z \lesssim \frac{1}{r}\norm[1]{g_1} \lesssim \frac{1}{r}\norm[1]{g}.
\end{align*} 
Also from part (1) of this lemma we have
\begin{align*}
\abs{v_2(\zone)}  \lesssim \norm[\infty]{g_2}^{2/3}\norm[1]{g_2}^{1/3} + \norm[1]{g_2}  \lesssim \norm[\infty]{g}^{2/3}\norm[1]{g_2}^{1/3} + \norm[1]{g_2}.
\end{align*}
Hence 
\begin{align*}
\sup_{\abs{x}\geq 2r} \abs{v(x)} \lesssim \frac{1}{r}\norm[1]{g} +  \norm[\infty]{g}^{2/3}\norm[1]{g_2}^{1/3} + \norm[1]{g_2}.
\end{align*}
As $\norm[1]{g_2} \to 0$ as $r\to \infty$, we are done. 

\item As $v(x) = \int_{\Omega}K_{\Omega}(x,y)g(y) \diff y$ and $ K_{\Omega}(x,y) = \grad_x^\perp G_{\Omega}(x,y)$, where $G_{\Omega}$ is the Green's function of $\Omega$, we see that $\grad \cdot v = 0$, $\grad \times v = g$ and $v\cdot n = 0$ on $\partial \Omega$. From part (2) of this lemma we have $v \in C(\Omegabar)$ and from part (3)  we also see that $\lim_{R \to \infty} \sup_{\abs{x} \geq R}\abs{v(x)} = 0$. Hence $v$ satisfies all the properties satisfied by $f$.

Now let $ p = f-v$. As $\grad \cdot p = 0$ and $\grad \times p = 0$ we see that $\myoverline{0}{0}{p}$ is a holomorphic function on $\Omega$. Let $P:\Hsp \to \Csp$ be defined as 
\begin{align*}
P(z) = \myoverline{0}{0}{p}(\Psi^{-1}(z))(\Psi^{-1})_z(z).
\end{align*}
Observe that $P$ is a holomorphic function on $\Hsp$. From \eqref{eq:Psiepzero} we see that for $z\in \Hsp$ we have $(\Psi^{-1})_z(z) = \nu z^{(\nu - 1)}$. Now as $p\cdot n =0$ on $\partial \Omega$ we see that $P$ is real valued on $\Rsp\backslash\cbrac{0}$. Hence by the Schwarz reflection principle, we can extend $P$ to a holomorphic function on $\Csp\backslash\cbrac{0}$. As $p \in C(\Omegabar)$ we see that $\lim_{z\to 0} zP(z) = 0$ and hence $P$ can be extended to a holomorphic function on $\Csp$. As $\lim_{R \to \infty} \sup_{\abs{x} \geq R}\abs{p(x)} = 0$, we see that $P$ is a bounded entire function on $\Csp$ which goes to $0$ at infinity. Hence $P = 0$ and so $p = 0$.  
\end{enumerate}
\end{proof}

Now let $(u_\ep,\w_\ep)$ be a Yudovich weak solution in $\Omega_\ep$ in the time interval $[0,T)$. For this solution, the flow $X_\ep:\Omega_\ep\times [0,T) \to \Omega_\ep$ is defined by \footnote{We see from \lemref{lem:dist} that this map is well defined.}
\begin{align}\label{eq:ODEX}
\frac{\diff X_{\ep}(x,t)}{\diff t} = u_\ep(X_\ep(x,t),t) \qq X_\ep(x,0) = x.
\end{align}
From  \lemref{lem:velocity} part (2) we see that the velocity is log-Lipschitz in the interior of $\Omega_\ep$ and hence this ODE can be solved uniquely as long as $X_\ep(x,t) \in \Omega_\ep$. We first recall the quantities related to the flow for the domain $\Omega_\ep$. We see from \eqref{eq:b} and \eqref{eq:X}  that
\begin{align}\label{eq:Xep}
\frac{\diff X_\ep(x,t)}{\diff t} = b_\ep(X_\ep(x,t),t)\myoverline{0}{0}{(\Psi_\ep)_z}(X_\ep(x,t)),
\end{align}
where $b_\ep:\Omega_\ep\times [0,\infty) \to \Csp$ is defined as
\begin{align}\label{eq:bep}
b_\ep(\zone,t)  :=  \brac*[\Big]{\frac{i}{2\pi}} \int_{\Omega_\ep}\sqbrac{\frac{1}{\Psibar_\ep(z_1) - \Psibar_\ep(z)} - \frac{1}{\Psibar_\ep(z_1) - \Psi_\ep(z)}}\w_\ep(z,t) \diff z.
\end{align}
Similarly we define the flow $Y_\ep:\Hsp\times[0,T) \to \Hsp $ as $Y_\ep(y,t) := \Psi_\ep(X_\ep(x,t))$, where $y = \Psi_\ep(x)$. Similarly define $\btil_\ep:\Hsp \times [0,\infty) \to \Csp$ as $\btil_\ep(y,t) := b_\ep(x,t)$. Hence from \eqref{eq:Y} we have
\begin{align}\label{eq:Yep}
\frac{\diff Y_\ep(y,t)}{\diff t} = \btil_\ep(Y_\ep(y,t),t)\abs{((\Psi_\ep)_z \compose \Psi_\ep^{-1}) (Y_\ep(y,t))}^{2}.
\end{align}
Defining $\wtil_\ep:\Hsp\times[0,\infty) \to \Rsp$ as $\wtil_\ep(s,t) := \w_\ep(z,t)$, where $s = \Psi_\ep(z)$, we get from \eqref{eq:btil} 
\begin{align}\label{eq:btilep}
\btil_\ep(y,t) =  \brac*[\Big]{\frac{i}{2\pi}} \int_{\Hsp}\sqbrac{\frac{1}{\ybar - \sbar} - \frac{1}{\ybar - s}}\wtil_\ep(s,t) \abs{((\Psi_\ep)_z \compose \Psi_\ep^{-1}) (s)}^{-2} \diff s.
\end{align}

We now show that the flow $X_\ep$ always remains in the domain $\Omega_\ep$ and hence the maps $X_\ep:\Omega_\ep\times [0,T) \to \Omega_\ep$ and $Y_\ep:\Hsp\times[0,T) \to \Hsp $ are well defined. The following lemma is analogous to similar statements proven in  \cite{La15, LaZl19,HaZl21}. 

\begin{lemma}\label{lem:dist}
Let $0\leq \ep\leq 1$ and let $(u_\ep,\w_\ep)$ be a Yudovich weak solution in the domain $\Omega_\ep$ in the time interval $[0,T)$ with initial vorticity $\w_0 \in \Lone(\Omega_\ep)\cap\Linfty(\Omega_\ep)$. Let $R>0$ and let $x_0\in \Omega_\ep$ with $\abs{x_0} \leq R$. Then there exists constants $c, C_1,C_2,C_3,C_4>0$ and $0<\ep_0\leq 1$ all depending only on  $R,T$ and  $\esssup_{t\in [0,T)}\norm[\Lone\cap\Linfty]{\w_\ep(\cdot,t)}$ so that if $0 \leq \ep \leq \ep_0$ then
\begin{align*}
 C_1\cbrac{\Imag(Y_\ep(y_0,0))}^{e^{ct}} \leq \Imag(Y_\ep(y_0,t))   \leq  C_2\cbrac{\Imag(Y_\ep(y_0,0))}^{e^{-ct}}
\end{align*}
and also
\begin{align*}
 C_3 d(X_\ep(x_0,0), \partial \Omega_\ep)^{\frac{1}{\nu} e^{ct}} \leq d(X_\ep(x_0,t), \partial \Omega_\ep) \leq C_4 d(X_\ep(x_0,0), \partial \Omega_\ep)^{\nu e^{-ct}}.
\end{align*} 
\end{lemma}
\begin{proof}
In this proof we will let $C>0$ denote a general constant which depends on $R,T$ and on $\esssup_{t\in [0,T)}\norm[\Lone\cap\Linfty]{\w_\ep(\cdot,t)}$ and we write $a \lesssim_C b$ instead of $a\leq C b$. 

We will first prove the estimate for $Y_\ep(y,t)$ and then translate that information into an estimate for $X_\ep(x,t)$. Let $y_0 = \Psi_\ep(x_0)$. Now as $u_\ep$ is bounded from \lemref{lem:velocity} we see from \eqref{eq:ODEX} that there exists $R_1>1$ depending only on $R, T$ and $ \esssup_{t\in [0,T)}\norm[\Lone\cap\Linfty]{\w_\ep(\cdot,t)}$ so that for all $t\in[0,T)$ and $0\leq \ep \leq 1$ we have
\begin{align}\label{eq:boundXepRone}
\abs{X_\ep(x_0,t)} \leq \frac{R_1}{2} \leq R_1.
\end{align}
Now we choose a $0<\ep_0\leq 1$ such that $\ep_0 R_1^{\frac{1}{\nu}} \leq 1/8$. Hence from \eqref{eq:Psiep} and \eqref{eq:Psiepzero} it is clear that if $0 \leq \ep \leq \ep_0$, then fo $t \in [0,T)$ we have
\begin{align}\label{eq:Yeplater}
\abs{Y_\ep(y_0,t)} = \abs{\Psi_\ep(X_\ep(x_0,t))}   \leq 4R_1^{\frac{1}{\nu}}.
\end{align}
Therefore $\ep\abs{Y_\ep(y_0,t)} \leq 1/2$. 

Now for $y\in \Hsp$ we see from \eqref{eq:btilep} that
\begin{align*}
\btil_\ep(y,t) =  \brac*[\Big]{\frac{i}{2\pi}} \int_{\Hsp}\sqbrac{\frac{1}{\ybar - \sbar} - \frac{1}{\ybar - s}}\wtil_\ep(s,t) \abs{((\Psi_\ep)_z \compose \Psi_\ep^{-1}) (s)}^{-2} \diff s.
\end{align*}
Therefore
\begin{align*}
& \Imag(\btil_\ep(y,t)) \\
& =  \brac*[\Big]{\frac{1}{2\pi}} \Real\cbrac{\int_{\Hsp}\sqbrac{\frac{1}{\ybar - \sbar} - \frac{1}{\Real(\ybar) - \sbar} + \frac{1}{\Real(\ybar) - s}  - \frac{1}{\ybar - s}}\wtil_\ep(s,t) \abs{((\Psi_\ep)_z \compose \Psi_\ep^{-1}) (s)}^{-2} \diff s}.
\end{align*}
Hence using part 4 of \lemref{lem:Psiep} we see that 
\begin{align*}
 \abs*{ \Imag(\btil_\ep(y,t))} & \lesssim \abs{\Imag(y)}\int_{\Hsp} \frac{\abs{\wtil_\ep(s,t)} \abs{s+i \ep}^{2\nu -2}}{\abs{y -s}\abs{\Real(y) - s}} \diff s \\
 & \lesssim \norm[\infty]{\w_\ep(\cdot,t)} \int_{\Rsp^2} \frac{\abs{\Imag(y)}}{\abs{(y+ i\ep) - s}\abs{(\Real(y) + i\ep) - s}\abs{s}^{2- 2\nu}} \diff s. 
\end{align*}
Observe that if we let $z_1 = y+i\ep$ and $z_2 = \Real(y) + i\ep$, then $\abs{\Imag(y)} = \abs{z_1 - z_2}$. Hence using the definition of $I$ from \eqref{def:I} we see that
\begin{align*}
\abs*{ \Imag(\btil_\ep(y,t))}  \lesssim \norm[\infty]{\w_\ep(\cdot,t)} \abs{z_1 - z_2}I((0,2-2\nu), (z_1,1), (z_2,1): (1,0,\infty)).
\end{align*}
Thus using the second estimate of \lemref{lem:Ithree} and observing that $\abs{y+i\ep} \geq \abs{\Real(y) + i\ep}$ we obtain
\begin{align*}
 \abs*{ \Imag(\btil_\ep(y,t))}  \lesssim \norm[\infty]{\w_\ep(\cdot,t)}   \brac{1 + \abs{y + i\ep}^{2\nu -2}}\phi(\abs{\Imag(y)}).
\end{align*}
Now from \eqref{eq:Yep} we get 
\begin{align*}
\frac{\diff \Imag\cbrac{Y_\ep(y_0,t)}}{\diff t} = \Imag\cbrac*[\big]{\btil_\ep(Y_\ep(y_0,t),t)}\abs{((\Psi_\ep)_z \compose \Psi_\ep^{-1}) (Y_\ep(y,t))}^{2}.
\end{align*}
Now as $\ep \abs{Y_\ep(y_0,t),t)} \leq 1/2$, we see from \eqref{eq:temp23} and \eqref{eq:Psizepzero} that
\begin{align*}
\abs{((\Psi_\ep)_z \compose \Psi_\ep^{-1}) (Y_\ep(y,t))}^{2} \lesssim \abs{Y_\ep(y_0,t) + i\ep}^{2-2\nu}.
\end{align*}
Hence
\begin{align}\label{eq:ImY}
\begin{split}
& \abs{\frac{\diff \Imag\cbrac{Y_\ep(y_0,t)}}{\diff t}} \\
& \lesssim \norm[\infty]{\w_\ep(\cdot,t)}  \phi(\abs{\Imag(Y_\ep(y_0,t))})\brac{1 + \abs{Y_\ep(y_0,t) + i\ep}^{2\nu -2}} \abs{Y_\ep(y_0,t) + i\ep}^{2-2\nu}  \\
& \lesssim \norm[\infty]{\w_\ep(\cdot,t)} \phi(\abs{\Imag(Y_\ep(y_0,t))})\brac{1 + \abs{Y_\ep(y_0,t) + i\ep}^{2 - 2\nu}}.
\end{split}
\end{align}
Now using  \eqref{eq:Yeplater} we get
\begin{align*}
\abs{\frac{\diff \Imag\cbrac{Y_\ep(y_0,t)}}{\diff t}} \lesssim_{C} \phi(\abs{\Imag(Y_\ep(y_0,t))}).
\end{align*}
Consequently from \lemref{lem:phi} there exists $c = c(R,T,\esssup_{t\in [0,T)}\norm[\Lone\cap\Linfty]{\w_\ep(\cdot,t)}) >0$ such that for all $t\in[0,T)$ we have
\begin{align}\label{eq:ImYest}
 \cbrac{\Imag(Y_\ep(y_0,0))}^{e^{ct}} \lesssim_{C} \Imag(Y_\ep(y_0,t))  \lesssim_{C}  \cbrac{\Imag(Y_\ep(y_0,0))}^{e^{-ct}}.
\end{align}
This proves the first part of the lemma. 

Now let $\tilde{x} \in \partial\Omega_\ep$ be such that $d(X_\ep(x_0,t), \partial\Omega_\ep) = d(X_\ep(x_0,t), \tilde{x})$. As the line segment joining the origin and $X_\ep(x_0,t)$ intersects $\partial \Omega_\ep$, we see that $\abs{\tilde{x}} \leq 2 \abs{X_\ep(x_0,t)}$. In particular we have $\abs{\tilde{x}} \leq R_1$ from \eqref{eq:boundXepRone}. Hence by the same argument used to show \eqref{eq:Yeplater}, we also see that $\abs{\Psi_\ep(\tilde{x})} \leq 4R_1^{\frac{1}{\nu}}$ and hence $\ep \abs{\Psi_\ep(\tilde{x})} \leq 1/2$. 

So now consider $z \in \Hsp$ with $\abs{\ep z} \leq 1/2$. We claim that for such $z $ we have
\begin{align}\label{eq:mappingbasicep1}
\abs{\frac{z}{1 - i\ep z} + i\ep} \approx \abs{z + i\ep}.
\end{align}
To see this consider first the case of $\abs{z} \geq 100 \ep$. Here it is clear that
\begin{align*}
\abs{\frac{z}{1 - i\ep z} + i\ep} \approx \abs{z} \approx \abs{z + i\ep}.
\end{align*}
Now if $\abs{z} \leq 100\ep$, then clearly 
\begin{align*}
\abs{\frac{z}{1 - i\ep z} + i\ep} \lesssim \ep \approx \abs{z + i\ep}.
\end{align*}
On the other hand we know that the map \eqref{eq:mapbasic} maps the upper half plane to the ball $B\brac{i\brac*[\Big]{\ep + \frac{1}{2\ep}}, \frac{1}{2\ep}}$. Hence 
\begin{align*}
\ep \leq \abs{\frac{z}{1 - i\ep z} + i\ep}.
\end{align*}
This proves the claim \eqref{eq:mappingbasicep1}. Also if $z_1, z_2 \in \Hsp$ are such that $\abs{\ep z_1}, \abs{\ep z_2} \leq 1/2$, then we see that
\begin{align}\label{eq:mappingbasicep2}
\abs{ \sqbrac{c(\ep) + \frac{1}{\ep(i + \ep\z_1)}} -  \sqbrac{c(\ep) + \frac{1}{\ep(i + \ep\z_2)}}} \approx \abs{z_1 - z_2}.
\end{align}

Therefore using \eqref{eq:Psiinvep}, \lemref{lem:powers}, \eqref{eq:mappingbasicep2} and \eqref{eq:mappingbasicep1} we get
\begin{align*}
& d(X_\ep(x_0,t), \partial\Omega_\ep) \\
& = \min_{\ep\abs{s} \leq \frac{1}{2}, s \in \Rsp}\abs{\Psi_\ep^{-1}(Y_\ep(y_0,t)) - \Psi_\ep^{-1}(s) } \\
& \approx  \min_{\ep\abs{s} \leq \frac{1}{2}, s \in \Rsp} \abs{Y_\ep(y_0,t) - s} \min\cbrac{\abs{\frac{Y_\ep(y_0,t)}{1 - i\ep Y_\ep(y_0,t)} + i\ep}^{\nu -1}, \abs{\frac{s}{1 - i\ep s} + i\ep}^{\nu -1} } \\
& \approx  \min_{\ep\abs{s} \leq \frac{1}{2}, s \in \Rsp}\abs{(Y_\ep(y_0,t) - s}\abs{(Y_\ep(y_0,t) + i\ep)}^{\nu-1} \\
& \approx  \abs{(Y_\ep(y_0,t) + i\ep)^\nu - \cbrac{\Real(Y_\ep(y_0,t)) + i\ep}^\nu }.
\end{align*}
Furthermore we see from \lemref{lem:powers} that
\begin{align*}
\cbrac{\Imag(Y_\ep(y_0,t))}\abs{Y_\ep(y_0,t) + i\ep}^{\nu -1}  \lesssim_{C} d(X_\ep(x_0,t), \partial\Omega_\ep)  \lesssim_{C} \cbrac{\Imag(Y_\ep(y_0,t))}^\nu.
\end{align*}
Hence
\begin{align*}
\cbrac{\Imag(Y_\ep(y_0,t))}   \lesssim_{C} d(X_\ep(x_0,t), \partial\Omega_\ep)  \lesssim_{C} \cbrac{\Imag(Y_\ep(y_0,t))}^\nu.
\end{align*}
In particular for $t=0$ we have
\begin{align*}
\cbrac{\Imag(Y_\ep(y_0,0))}   \lesssim_{C} d(X_\ep(x_0,0), \partial\Omega_\ep)  \lesssim_{C} \cbrac{\Imag(Y_\ep(y_0,0))}^\nu.
\end{align*}
Combing these two estimates with \eqref{eq:ImYest} we obtain for all $t\in[0,T)$ 
\begin{align*}
 d(X_\ep(x_0,0), \partial \Omega_\ep)^{\frac{1}{\nu} e^{ct}} \lesssim_{C} d(X_\ep(x_0,t), \partial \Omega_\ep)  \lesssim_{C} d(X_\ep(x_0,0), \partial \Omega_\ep)^{\nu e^{-ct}}.
\end{align*}
\end{proof}

We are now ready to prove the existence of Yudovich weak solutions in $\Omega$. 
\begin{thm}\label{thm:weak}
Consider an initial data $(u_0,\w_0)$ satisfying \eqref{eq:initial} in the domain $\Omega$. Then there exists a Yudovich weak solution $(u,\w)$ in domain $\Omega$ in the time interval $[0,\infty)$ in the sense of \eqref{eq:weak} and \eqref{eq:divweak}.  
\end{thm}
\begin{proof}
We closely follow the existence proof of weak solutions in $\Rsp^2$ as given in Chapter 8 of \cite{MaBe02}. Observe that it is enough to prove the existence in the time interval $[0,T)$ for arbitrary $T>0$. By restricting $\w_0$ to compact sets and by convolution, we see that there exists initial vorticities $(\w_0)_\ep \in C^{\infty}_c(\Omega_\ep) \subset C^{\infty}_c(\Omega)$ such that for all $0<\ep\leq 1$
\begin{align*}
\norm[\Linfty(\Omega_\ep)]{(\w_0)_\ep} \leq \norm[\Linfty(\Omega)]{\w_0}, \qq \norm[\Lone(\Omega_\ep)]{(\w_0)_\ep} \leq \norm[\Lone(\Omega)]{\w_0}
\end{align*}  
and 
\begin{align*}
\norm[\Lone(\Omega)]{(\w_0)_\ep - \w_0} \to 0 \qq \tx{ as }\ep \to 0.
\end{align*}
Now for $\ep>0$ the domain $\Omega_\ep$ is a smooth bounded domain and hence there exists a unique smooth solution $(u_\ep,\w_\ep)$ in $\Omega_\ep$ in the time interval $[0,T)$ with initial vorticity $(\w_0)_\ep$ (see \cite{MaPu94}). Let the corresponding flows be $X_\ep:\Omega_\ep \times [0,T) \to \Omega_\ep$, then from the transport equation we see that $\w_\ep(x,t) = (\w_0)_\ep(X_\ep^{-1}(x,t))$. As $X_\ep(\cdot,t)$ and $X_\ep^{-1}(\cdot,t)$ are measure preserving, we see that for all $1\leq p\leq \infty$ we have $\norm[L^p(\Omega_\ep)]{\w_\ep(\cdot,t)} = \norm[L^p(\Omega_\ep)]{(\w_0)_\ep} \leq \norm[L^p(\Omega)]{\w_0}$. 

\medskip
\noindent\textbf{Step 1:} Let $K \subset \Omega$ be a compact set and let $R>0$ be such that $\abs{x} \leq R$ for all $x \in K$. Let $0<\ep_0\leq 1$ be such that $K \subset \Omega_\ep$ for all $0<\ep\leq \ep_0$ and $\ep_0$ also satisfies the conditions of \lemref{lem:dist}. From using the fact that the velocity is uniformly bounded by \lemref{lem:velocity} followed by \lemref{lem:dist}, we see that there exists  a compact set $K_1 \subset \Omega$ such that for all $0<\ep\leq \ep_0$ and $x \in K$ and $t_1,t_2 \in [0,T)$ we have that $X_\ep(X_\ep^{-1}(x,t_1),t_2), X_\ep^{-1}(X_\ep(x,t_1),t_2) \in K_1$. Similarly there also exists compact sets $K_2,K_3 \subset \Omega$ such that for all $0<\ep\leq \ep_0$ and $t_1,t_2 \in [0,T)$ we have for $x \in K_1$
\begin{align*}
X_\ep(X_\ep^{-1}(x,t_1),t_2), X_\ep^{-1}(X_\ep(x,t_1),t_2) \in K_2
\end{align*}
and similarly for $x \in K_2$ we have
\begin{align*}
X_\ep(X_\ep^{-1}(x,t_1),t_2), X_\ep^{-1}(X_\ep(x,t_1),t_2) \in K_3.
\end{align*}
These sets will be useful to prove estimates for the maps $X$ and $X^{-1}$ below. Now from \lemref{lem:velocity} observe that for all $\zone,\ztwo \in K_2$ and $0<\ep\leq \ep_0$ we have 
\begin{align*}
\frac{\diff \abs{X_\ep(\zone,t) - X_\ep(\ztwo,t)}}{\diff t} & \leq \abs{u_\ep(X_\ep(\zone,t),t) - u_\ep(X_\ep(\ztwo,t),t)} \\
& \leq C_{K_3, \ep_0,  \norm[\Lone\cap\Linfty]{\w_0}} \phi( \abs{X_\ep(\zone,t) - X_\ep(\ztwo,t)}).
\end{align*}
Hence from \lemref{lem:phi} there exists $\gamma_1,\gamma_2>0$ depending only on  $K_3, \ep_0, T$  and $ \norm[\Lone\cap\Linfty]{\w_0}$ so that for all $\zone,\ztwo \in K_2$ and $t\in[0,T)$
\begin{align}\label{eq:Xepest}
\abs{\zone - \ztwo}^{\gamma_1} \leq \abs{X_\ep(\zone,t) - X_\ep(\ztwo,t)} \leq \abs{\zone - \ztwo}^{\gamma_2}.
\end{align}
Therefore for all $\zone,\ztwo \in K_1$ and $t\in[0,T)$
\begin{align}\label{eq:Xepinvest}
 \abs{\zone - \ztwo}^{\frac{1}{\gamma_2}} \leq \abs{X_\ep^{-1}(\zone,t) - X_\ep^{-1}(\ztwo,t)} \leq \abs{\zone - \ztwo}^{\frac{1}{\gamma_1}}.
\end{align}
As the velocity is bounded by \lemref{lem:velocity}, we have for all $x\in K$ and all $t_1,t_2 \in [0,T)$
\begin{align*}
\abs{X_\ep(x,t_1) - X_\ep(x,t_2)}  \lesssim_{K, \ep_0, \norm[\Lone\cap\Linfty]{\w_0}} \abs{t_1 - t_2}.
\end{align*}
Now let $X_\ep^*(x,t;\tau) = X_\ep(X_\ep^{-1}(x,t),t-\tau)$ denote  the backward particle trajectories with $X_\ep^*(x,t;t) = X_\ep^{-1}(x,t)$ and which satisfies the ODE
\begin{align*}
\frac{\diff X_\ep^*(x,t;\tau)}{\diff \tau} = -u_\ep(X_\ep^*(x,t;\tau), t-\tau) \qq \qq X_\ep^*(x,t;0) = x.
\end{align*}
Observe that for $x\in K$ we have $X_\ep^*(x,t;\tau) \in K_1$. Hence from the above equation, \eqref{eq:Xepinvest} and \lemref{lem:velocity} we see that for all $x\in K$ and all $0\leq t_1 \leq t_2 <T$ we have
\begin{align*}
\abs{X_\ep^{-1}(x,t_1) - X_\ep^{-1}(x,t_2)} & = \abs{X_\ep^{-1}(x,t_1) - X_\ep^{-1}(X_\ep^*(x,t_2;t_2-t_1),t_1)} \\
& \leq \abs{x - X_\ep^*(x,t_2;t_2-t_1)}^{\frac{1}{\gamma_1}} \\
&  \lesssim_{\norm[\Lone\cap\Linfty]{\w_0}}   \abs{t_2-t_1}^{\frac{1}{\gamma_1}}.
\end{align*}

\medskip
\noindent\textbf{Step 2:} Using these estimates we see that for $0<\ep\leq \ep_0$ the restricted functions $X_\ep,X_\ep^{-1}:K\times[0,T) \to \Omega$ form equicontinuous families. Hence by Arzela Ascoli and a diagonalization argument and passing to a subsequence we get continuous functions $X,X^{-1}:\Omega\times[0,T) \to \Omega$ such that 
\begin{align*}
X_\ep \to X  \qq\tx{ and }\quad X_\ep^{-1} \to X^{-1}
\end{align*}
uniformly on compact subsets of $\Omega\times [0,T)$. Hence for all $t\in [0,T)$ the function $X(\cdot,t):\Omega\to \Omega$ is a homeomorphism. As $X_\ep(\cdot,t)$ and $X_\ep(\cdot,t)$ are measure preserving, we see that for any $f \in C_c(\Omega_\ep) \subset C_c(\Omega)$
\begin{align*}
\int_{\Omega} f(X_\ep(x,t))1_{x\in \Omega_\ep} \diff x = \int_{\Omega_\ep} f(X_\ep(x,t)) \diff x =  \int_{\Omega_\ep} f(x) \diff x = \int_{\Omega} f(x) \diff x.
\end{align*}
Hence by letting $\ep \to 0$ and by an approximation argument we see that $X(\cdot,t)$ and $X^{-1}(\cdot,t)$ are also measure preserving. 

We can finally define $\w:\Omega \times [0,T) \to \Rsp$ as $\w(x,t) := \w_0(X^{-1}(x,t))$ and $u(x,t) := \int_{\Omega}K_{\Omega}(x,y)\w(y,t) \diff y$. It is easy to see that $(u,\w)$ is in the Yudovich class \eqref{eq:Yudovich} by using that fact that $X^{-1}(\cdot,t)$ is measure preserving and from \lemref{lem:velocity}. 

Let us extend the function $\w_\ep(\cdot,t) : \Omega_\ep\to \Rsp$ to $\w_\ep(\cdot,t) : \Omega\to \Rsp$ by zero. We then claim that for any $t\in [0,T)$ we have
\begin{align*}
\norm[\Lone(\Omega)]{\w_\ep(\cdot,t) - \w(\cdot,t)} \to 0 \qq \tx{ as } \ep \to 0.
\end{align*}
To see this observe that for any fixed $0<\ep_0 < 1$, for all $0<\ep\leq \ep_0$ and $x \in \Omega_{\ep_0}$ we have
\begin{align*}
& \abs{\w_\ep(x,t) - \w(x,t)} \\
& = \abs{(\w_0)_\ep(X_\ep^{-1}(x,t)) - \w_0(X^{-1}(x,t))} \\
& \leq  \abs{(\w_0)_\ep(X_\ep^{-1}(x,t)) - (\w_0)_{\ep_0}(X_\ep^{-1}(x,t))} +  \abs{(\w_0)_{\ep_0}(X_\ep^{-1}(x,t)) - (\w_0)_{\ep_0}(X^{-1}(x,t))} \\
& \quad + \abs{(\w_0)_{\ep_0}(X^{-1}(x,t)) - \w_0(X^{-1}(x,t))}.
\end{align*}
Hence using the fact that $\norm[\Lone(\Omega)]{(w_0)_\ep - \w_0} \to 0$ as $\ep \to 0$, $X_\ep^{-1}(\cdot,t)$ and $X^{-1}(\cdot,t)$ are measure preserving and the fact that $X_\ep^{-1}(\cdot,t) \to X^{-1}(\cdot,t)$ uniformly on compact subsets of $K$, we see that
\begin{align*}
\norm[\Lone(\Omega)]{(\w_\ep(\cdot,t) - \w(\cdot,t))\one_{\Omega_\ep}} \to 0 \qq \tx{ as } \ep \to 0.
\end{align*}
As $\norm[\Lone(\Omega)]{\w(\cdot,t)\one_{\Omega_\ep} - \w(\cdot,t)} \to 0$ as $\ep \to 0$, the claim is proved. 

We extend $u_\ep(\cdot,t):\Omega_\ep \to \Csp$ to $u_\ep(\cdot,t):\Omega \to \Csp$ by zero. We now claim that for any fixed $t \in [0,T)$ we have $u_\ep(x,t) \to u(x,t)$ a.e. $x \in \Omega$. To see this observe that we have for $\zone \in \Omega_\ep$
\begin{align*}
u_\ep(\zone,t)   =   \int_{\Omega_\ep} K_{\Omega_\ep}(z_1, z) \w_\ep(z,t) \diff z.
\end{align*}
As $\w_\ep(\cdot,t) = 0$ on $\Omega\backslash \Omega_\ep$ we have
\begin{align*}
u_\ep(\zone,t)   =   \int_{\Omega} K_{\Omega_\ep}(z_1, z) (\w_\ep(z,t) - \w(z,t)) \diff z +  \int_{\Omega} K_{\Omega_\ep}(z_1, z)  \w(z,t) \diff z.
\end{align*}
Using \lemref{lem:Psiep} we see that the second term converges to $u(\zone,t)$ by dominated convergence. The first term can be easily controlled by a similar computation as done in \lemref{lem:velocity}; that is
\begin{align*}
& \abs{  \int_{\Omega} K_{\Omega_\ep}(z_1, z) (\w_\ep(z,t) - \w(z,t)) \diff z} \\
& \lesssim \int_{\Omega} \frac{1}{\abs{\zone - z}} \abs{\w_\ep(z,t) - \w(z,t)}\diff z \\
& \lesssim \norm[L^3(\Omega\cap B_1(z_1))]{\w_\ep(\cdot,t) - \w(\cdot,t)} + \norm[L^1(\Omega)]{\w_\ep(\cdot,t) - \w(\cdot,t)}
\end{align*}
which goes to $0$ as $\ep \to 0$. 

\medskip
\noindent\textbf{Step 3:} Let us now show that $(u,\w)$ is a weak solution to the Euler equation \eqref{eq:weak}. Let $\varphi \in C^{\infty}_c(\Omega\times [0,T))$. Then there exists $\ep_0>0$ such that $\supp(\varphi) \subset \Omega_{\ep_0}\times [0,T)$. Hence for all $0<\ep\leq \ep_0$ we see that 
\begin{align*}
\int_0^T \int_{\Omega} \w_\ep(\pt \varphi + u_\ep\cdot \grad \varphi) \diff x \diff t = - \int_{\Omega} (\w_0)_\ep \varphi(\cdot,0) \diff x.
\end{align*}
Now observe that
\begin{align*}
\int_0^T \int_{\Omega} \w_\ep(\pt \varphi + u_\ep\cdot \grad \varphi) \diff x \diff t & = \int_0^T \int_{\Omega} (\w_\ep - \w)(\pt \varphi + u_\ep\cdot \grad \varphi) \diff x \diff t \\
& \quad + \int_0^T \int_{\Omega} \w(\pt \varphi + u_\ep\cdot \grad \varphi) \diff x \diff t.
\end{align*}
The second term converges to 
\begin{align*}
\int_0^T \int_{\Omega} \w(\pt \varphi + u\cdot \grad \varphi) \diff x \diff t
\end{align*}
by dominated convergence. The first term can be controlled by using the fact that $u_{\ep}$ are bounded by \lemref{lem:velocity}
\begin{align*}
\abs{ \int_0^T \int_{\Omega} (\w_\ep - \w)(\pt \varphi + u_\ep\cdot \grad \varphi) \diff x \diff t } \lesssim_{\varphi,\norm[\Lone\cap\Linfty]{\w_0}} \int_0^T \norm[\Lone(\Omega)]{\w_\ep(\cdot,t) - \w(\cdot,t)} \diff t
\end{align*}
which goes to zero by dominated convergence. We also see that as $(\w_0)_\ep \to \w_0$ in $\Lone(\Omega)$ we have
\begin{align*}
- \int_{\Omega} (\w_0)_\ep \varphi(\cdot,0) \diff x  \to - \int_{\Omega} (\w_0) \varphi(\cdot,0) \diff x.
\end{align*}
Thus $(u,\w)$ satisfies \eqref{eq:weak}. Now for any $h \in G_c(\Omega)$ we see that
\begin{align*}
\int_{\Omega_\ep} u_\ep(\cdot,t)\cdot h = \int_{\Omega} u_\ep(\cdot,t)\cdot h = 0.
\end{align*}
Consequently by using the fact that $u_\ep$ are bounded by \lemref{lem:velocity}, we get from dominated convergence that
\begin{align*}
 \int_{\Omega} u(\cdot,t)\cdot h = 0.
\end{align*}
Hence proved. 
\end{proof}

\begin{lemma}\label{lem:transport}
Let $(u,\w)$ be a Yudovich weak solution with initial vorticity $\w_0$ in the domain $\Omega$ in the time interval $[0,T)$. Then 
\begin{enumerate}
\item The map $X(\cdot,t):\Omega \to \Omega$ is a homeomorphism for each $t\in [0,T)$ and the functions $X,X^{-1}:\Omega\times[0,T) \to \Omega$ are continuous. 
\item $\w(x,t) = \w_0(X^{-1}(x,t))$ for a.e. $(x,t) \in \Omega\times[0,T)$
\item If $(t_n)_{n=1}^\infty$ is a sequence in $[0,T)$ with $t_n \to t \in [0,T)$, then $\norm[1]{\w(\cdot,t_n) - \w(\cdot,t)} \to 0$ as $n\to \infty$. 
\item The functions $b:\Omegabar\times[0,T) \to \Csp$, $\btil:\Hspbar \times [0,T) \to \Csp$ and $u:\Omegabar\times[0,T) \to \Csp$ are bounded continuous functions and the ODE \eqref{eq:ODEX} for $\ep=0$ is true pointwise for all $(x,t) \in \Omega\times[0,T)$.
\end{enumerate}
\end{lemma}

\begin{proof}
We prove the statements sequentially.
\begin{enumerate}
\item As $X$ is defined as the solution to the ODE \eqref{eq:ODEX} for $\ep=0$ and as the velocity is locally log-Lipschitz from \lemref{lem:velocity}, we see that $X$ is continuous as long as $X(x,t) \in \Omega$. Now from \lemref{lem:dist} we see that $X(x,t) \in \Omega$ for all $(x,t) \in \Omega \times [0,T)$ and hence $X:\Omega\times[0,T) \to \Omega$ is continuous. From the same argument as the one used to derive \eqref{eq:Xepest} we see that $X(\cdot,t):\Omega\to \Omega$ is one to one. Now using this together with \lemref{lem:velocity} and \lemref{lem:dist}, we see that $X(\cdot,t)$ is onto $\Omega$ and hence $X(\cdot,t):\Omega \to \Omega$ is a homeomorphism. Hence $X^{-1}:\Omega\times[0,T) \to \Omega$ is also continuous. 

\item By using Lemma 3.1 in \cite{HaZl21} by Han and Zlato\v{s}, we directly get that $\w(x,t) = \w_0(X^{-1}(x,t))$ for a.e. $(x,t) \in \Omega\times[0,T)$. 

\item If $K \subset \Omega$ is a compact set then by a similar argument as the one used in \thmref{thm:weak} we see that the restricted functions $X^{-1}(\cdot,t_n):K \to \Omega$ form an equicontinuous family. As $X^{-1}(\cdot,t_n) \to X^{-1}(\cdot,t)$ pointwise, this implies that $X^{-1}(\cdot,t_n) \to X^{-1}(\cdot,t)$ uniformly on compact sets of $\Omega$. We now get that $\norm[1]{\w(\cdot,t_n) - \w(\cdot,t)} \to 0$ by approximating $\w_0$ by a function $g_\ep \in C_c(\Omega)$ in $\Lone(\Omega)$ and passing to the limit. 

\item  Recall that $b$ is given by the formula \eqref{eq:bep} with $\ep=0$. From a similar computation as in \lemref{lem:velocity} we see that
\begin{align*}
\norm[\Linfty(\Omega)]{b} \lesssim \norm[L^1(\Omega)\cap L^\infty(\Omega)]{\w_0}.
\end{align*}
Now if $(\zone,t_1), (\ztwo,t_2) \in \Omegabar\times[0,T)$ then from \lemref{lem:powers} and the calculations of \lemref{lem:velocity} we see that
\begin{align*}
 \abs{b(\zone,t_1) - b(\ztwo,t_2)} & \leq \abs{b(\zone,t_1) - b(\ztwo,t_1)} + \abs{b(\ztwo,t_1) - b(\ztwo,t_2)} \\
& \lesssim  \phi(\abs{\zone - \ztwo})\min\cbrac{\abs{\zone}^{1 - \frac{1}{\nu} },\abs{\ztwo}^{1 - \frac{1}{\nu} } }\norm[\Lone\cap\Linfty]{\w_0}  \\
& \quad + \int_{\Omega}\frac{1}{\abs{\ztwo - z}\abs{z}^{\frac{1}{\nu} - 1}}\abs{\w(z,t_1) - \w(z,t_2)}\diff z.
\end{align*}
For the first term we observe from \lemref{lem:powers} that
\begin{align*}
 \abs{z_1 - z_2}\min\cbrac{\abs{\zone}^{1 - \frac{1}{\nu} },\abs{\ztwo}^{1 - \frac{1}{\nu} } } & \approx \abs{z_1 - z_2}\min\cbrac{\abs{\zone}^{1 - \frac{1}{\nu} },\abs{\ztwo}^{1 - \frac{1}{\nu} }, \abs{z_1 - z_2}^{1 - \frac{1}{\nu} } }  \\
& \lesssim \abs{z_1 - z_2}^{2 - \frac{1}{\nu}}.
\end{align*}
Now by using the weighted AM-GM inequality and using the definition of $\phi$ from \eqref{def:phi} we get
\begingroup
\allowdisplaybreaks
\begin{align*}
& \abs{b(\zone,t_1) - b(\ztwo,t_2)} \\
& \lesssim \max\cbrac{-\ln(\abs{\zone - \ztwo}), 1}\abs{\zone - \ztwo}^{2-\frac{1}{\nu}}\norm[\Lone\cap\Linfty]{\w_0}  \\*
& \quad + \int_{\Omega}\frac{1}{\abs{\ztwo - z}^{\frac{1}{\nu}}}\abs{\w(z,t_1) - \w(z,t_2)}\diff z  + \int_{\Omega}\frac{1}{\abs{z}^{\frac{1}{\nu}}}\abs{\w(z,t_1) - \w(z,t_2)}\diff z \\
& \lesssim \max\cbrac{-\ln(\abs{\zone - \ztwo}), 1}\abs{\zone - \ztwo}^{2-\frac{1}{\nu}}\norm[\Lone\cap\Linfty]{\w_0} + \norm[L^{\frac{2\nu+1}{2\nu -1}}]{\w(\cdot,t_1) - \w(\cdot,t_2)} \\*
& \quad + \norm[\Lone]{\w(\cdot,t_1) - \w(\cdot,t_2)}.
\end{align*}
\endgroup
Hence $b:\Omegabar\times[0,T) \to \Csp$ is bounded and continuous and hence $\btil:\Hspbar\times[0,T) \to \Csp$ is also bounded and continuous.  From \lemref{lem:velocity} we already know that $u$ is bounded and from \eqref{eq:Kepzero} and \eqref{eq:bep} we see that $u(z,t) = \frac{1}{\nu}\zbar^{\frac{1}{\nu}-1}b(z,t)$ and hence $u:\Omegabar\times[0,T) \to \Csp$ is also continuous. As the velocity is continuous on $\Omega\times[0,T)$, the ODE \eqref{eq:ODEX} for $\ep=0$ is true pointwise for all $(x,t) \in \Omega\times[0,T)$.
\end{enumerate}
\end{proof}

\subsection{Properties of the flow}

From now on we will only consider flows on the domain $\Omega$ and so we will only be concerned with equations \eqref{eq:ODEX}, \eqref{eq:Xep},  \eqref{eq:bep} , \eqref{eq:Yep} and $\eqref{eq:btilep}$ for $\ep=0$. In particular we get that  $X:\Omega\times [0,T) \to \Omega$ satisfies
\begin{align}\label{eq:Xnew}
\frac{\diff X(x,t)}{\diff t} = \frac{1}{\nu}b(X(x,t),t)\Xbar(x,t)^{\frac{1}{\nu} -1},
\end{align}
where $b:\Omega\times [0,\infty) \to \Csp$ is given by
\begin{align}\label{eq:bnew}
b(\zone,t)  :=  \brac{\frac{i}{2\pi}} \int_{\Omega} \sqbrac{\frac{1}{\zonebar^{\frac{1}{\nu}} - \zbar^{\frac{1}{\nu}}} - \frac{1}{\zonebar^{\frac{1}{\nu}} - \z^{\frac{1}{\nu}}   } }\w(z,t) \diff z.
\end{align}
The flow $Y:\Hsp\times[0,T) \to \Hsp $ is defined as $Y(y,t) := X(x,t)^{\frac{1}{\nu}}$, where $y = x^\frac{1}{\nu}$. Similarly $\btil:\Hsp \times [0,\infty) \to \Csp$ is defined as $\btil(y,t) := b(x,t)$ and we have 
\begin{align}\label{eq:Ynew}
\frac{\diff Y(y,t)}{\diff t} = \frac{1}{\nu^2}\btil(Y(y,t),t)\abs{Y(y,t)}^{2-2\nu}.
\end{align}
Defining $\wtil:\Hsp\times[0,\infty) \to \Rsp$ as $\wtil(s,t) := \w(z,t)$, where $s = z^{\frac{1}{\nu}}$, we get 
\begin{align}\label{eq:btilnew}
\btil(y,t) =  \brac*[\Big]{\frac{i\nu^2}{2\pi}} \int_{\Hsp}\sqbrac{\frac{1}{\ybar - \sbar} - \frac{1}{\ybar - s}}\wtil(s,t) \abs{s}^{2\nu -2} \diff s.
\end{align}

 Let the initial vorticity $\w_0 \in \Lone(\Omega)\cap\Linfty(\Omega)$. Note that if $\w_0 \equiv 0$ then the flow is trivial and hence we assume that $\w_0 \not\equiv 0$. Observe that
\begin{align*}
b(0,0) = \btil(0,0) =  \brac*[\Big]{\frac{i\nu^2}{2\pi}} \int_{\Hsp}\sqbrac*[\bigg]{\frac{-1}{ \sbar } + \frac{1}{s}} \wtil_0(s) \abs{s}^{2\nu -2} \diff s = \frac{\nu^2}{\pi}\int_\Hsp \Imag(s)\wtil_0(s) \abs{s}^{2\nu-4} \diff s.
\end{align*}
Define 
\begin{align}\label{eq:bzero}
b_0 := b(0,0).
\end{align}
Hence if $\w_0 \geq 0$ and $\w_0 \not\equiv 0$, then $0<b_0<\infty$. (Recall that $\Hsp$ is the upper half plane and $\wtil_0(s) = \w_0(s^\nu)$ for $s \in \Hsp$). 

The next proposition quantifies the property that the support of the vorticity moves away from the corner for a short period of time. This is proved in part (3) of the proposition below. This is the analog of step (1) of the proof of the uniqueness of the ODE \eqref{eq:sampleODE} in the introduction. 

\begin{prop}\label{prop:flow}
Let $(u,\w)$ be a Yudovich weak solution in the domain $\Omega$ in the time interval $[0,\infty)$ with initial vorticity $\w_0 \in \Lone(\Omega)\cap\Linfty(\Omega)$ and assume that we have $b_0 >0$ where $b_0$ is defined in \eqref{eq:bzero}. Let $X:\Omega\times[0,\infty) \to \Omega$ be the flow map of the solution. Let $\ep>0$ be such that $0<\ep<\min\cbrac*[\big]{b_0,1}$. Then there exists $T>0$ and $0<R<1/10$ such that for all $t\in [0,T]$ we have
\begin{enumerate}
\item For all $x\in \Omega\cap B_{2R}(0) $ we have $\abs{b(x,t) - b_0} <\ep$.
\item  For all $x \in \Omega \cap B_R(0)$ we have $\abs{X(x,t)} \leq \frac{3R}{2}$ and for all $x \in \Omega \cap B_R(0)^c$ we have $\abs{X(x,t)} \geq \frac{R}{2}$.
\item For all $x\in \Omegaplus\cap B_R(0)$ we have $X(x,t) \in \Omegaplus$ and  
\begin{align*}
\abs{X(x,t)} \geq \sqbrac{\frac{(2\nu-1)(b_0 - \ep)t}{\nu^2}}^{\frac{\nu}{2\nu-1}}.
\end{align*}
\item For all $x\in \Omegaplus\cap B_R(0)$ we have $\abs{X(x,t)} \gtrsim \abs{x}^{1+\ep}$ 
\end{enumerate}
\end{prop}

\begin{proof} 
We will define $T>0$ at the very end of the proof. We will prove the result by proving the corresponding result for the flow $Y(y,t) = X(x,t)^{\frac{1}{\nu}}$ in the upper half plane. 
\begin{enumerate}
\item From \lemref{lem:transport} we know that $b:\Omegabar\times[0,\infty) \to \Csp$ is a continuous function. Hence there exists  $T_1>0$ and $0<R<1/10$ such that $\abs{b(x,t) - b_0} <\ep$ for all $x\in \Omega\cap B_{2R}(0) $ and $t\in [0,T_1]$. 

Now let $R^* := (R)^\frac{1}{\nu}$ and hence $0< R^* < 1/10$ and we have  $\abs*{\btil(y,t) - b_0} <\ep$ for all $y\in \Hsp\cap B(0,2^{\frac{1}{\nu}}R^*) $ and $t\in [0,T_1]$, where $\btil$ was defined in \eqref{eq:btilnew}. 

\item As the velocity is bounded from \lemref{lem:velocity}, there exists a constant $C_1>0$ such that $\abs{\frac{\diff X(x,t)}{\diff t}} \leq C_1$ for all $x\in \Omega$ and $t\in[0,\infty)$. Letting $T_2 = \frac{R}{2C_1} >0$ we see that for all $t\in [0,T_2]$ we have
\begin{align*}
\abs{X(x,t) - X(x,0)} \leq C_1 t \leq \frac{R}{2}.
\end{align*}
Now as $X(x,0) = x$ we see that for all $x \in \Omega \cap B_R(0)^c$ we have $\abs{X(x,t)} \geq \frac{R}{2}$. Similarly for all $x \in \Omega \cap B_R(0)$ we have $\abs{X(x,t)} \leq \frac{3}{2}R \leq 1/5$. 

Using this we see that for all $y \in \Hsp \cap B_{R^*}(0)^c$ and $t\in[0,T_2]$ we have  $\abs{Y(y,t)} \geq (\frac{R}{2})^{\frac{1}{\nu}} = R^*/2^{\frac{1}{\nu}}$.  Similarly we have  $\abs{Y(y,t)} \leq (\frac{3}{2}R)^{\frac{1}{\nu}}  \leq 1/5$ for all $y \in \Hsp \cap B_{R^*}(0)$ and $t\in[0,T_2]$.

\item Let $T_3 = \min\cbrac{T_1,T_2}>0$. From \eqref{eq:Ynew} and part (1) and (2) of this proposition we see that  for all $y\in \Hspplus\cap B_{R^*}(0)$ and $t\in[0,T_3]$, we have
\begin{align*}
\frac{\diff \Real\cbrac{Y(y,t)}}{\diff t} = \frac{1}{\nu^2}\Real(\btil(Y(y,t),t))\abs{Y(y,t)}^{2-2\nu} \geq \frac{(b_0 - \ep)}{\nu^2}\abs{\Real{Y(y,t)}}^{2-2\nu}  \geq 0.
\end{align*}
This says that the particle is moving to the right and hence $Y(y,t) \in \Hspplus$. We can quantify exactly how much it moves to the right by integrating and so
\begin{align*}
\frac{\Real\cbrac{Y(y,t)}^{2\nu-1}}{2\nu-1} - \frac{\Real\cbrac{Y(y,0)}^{2\nu-1}}{2\nu-1} \geq \frac{(b_0 - \ep)}{\nu^2}t.
\end{align*}
As $\Real\cbrac{Y(y,0)} = \Real(y) \geq 0$ for all $y\in \Hspplus\cap B_{R^*}(0)$ we obtain
\begin{align}\label{eq:ReYest}
\abs{Y(y,t)} \geq \Real\cbrac{Y(y,t)} \geq  \sqbrac{\frac{(2\nu-1)(b_0 - \ep)t}{\nu^2}}^{\frac{1}{2\nu-1}}.
\end{align}
Hence $\abs{X(x,t)} \geq \sqbrac{\frac{(2\nu-1)(b_0 - \ep)t}{\nu^2}}^{\frac{\nu}{2\nu-1}}$ for all $x\in \Omegaplus\cap B_R(0)$.

\item Let $x\in \Omega\cap B_R(0)$ and $y = x^{\frac{1}{\nu}}$. As $\abs{Y(y,t)} \lesssim 1$ for all $y \in \Hsp\cap B_{R^*}(0)$ and $t\in [0,T_3]$, we see from \eqref{eq:ImY} in \lemref{lem:dist}  that 
\begin{align*}
 \abs{\frac{\diff \Imag\cbrac{Y(y,t)}}{\diff t}} &  \lesssim \phi(\abs{\Imag(Y(y,t))})\brac{1 + \abs{Y(y,t)}^{2 - 2\nu}}\norm[\infty]{\w_0} \\
 &  \lesssim \phi(\abs{\Imag(Y(y,t))})\norm[\infty]{\w_0}.
\end{align*}
Hence from \lemref{lem:phi} we see that there exists $C_2 = C_2(\norm[\Lone\cap\Linfty]{\w_0})>0$ so that for all $y \in \Hsp\cap B_{R^*}(0)$ and $t\in [0,T_3]$ we have 
\begin{align*}
 \Imag(Y(y,t))  \gtrsim  \cbrac{\Imag(Y(y,0))}^{e^{C_2t}}.
\end{align*}
Let $T_4>0$ be such that $1<e^{C_2 T_4} < 1+\ep$ and let $T_5 = \min\cbrac{T_3, T_4}>0$. Then for all $y\in \Hsp\cap B_{R^*}(0)$ in the time $t\in [0,T_5]$ we have
\begin{align*}
 \Imag\cbrac{Y(y,t)} \gtrsim  \brac{\Imag\cbrac{Y(y,0)}}^{1+\ep}.
\end{align*}

We can now prove the required estimate. For $y\in \Hspplus\cap B_{R^*}(0)$ satisfying $\Real(y) \leq \Imag(y)$ we see that for all $t\in [0,T_5]$ we have
\begin{align}\label{eq:part4int1}
\abs{Y(y,t)} \geq \Imag\cbrac{Y(y,t)} \gtrsim  \brac{\Imag\cbrac{Y(y,0)}}^{1+\ep} \gtrsim \brac{\Imag(y)}^{1+\ep} \gtrsim \abs{y}^{1+\ep}.
\end{align}
For $y\in \Hspplus\cap B_R(0)$ satisfying $\Real(y) \geq \Imag(y)$ we see that for all $t\in [0,T_5]$ we have
\begin{align*}
\frac{\diff \Real\cbrac{Y(y,t)}}{\diff t} \geq 0.
\end{align*}
Hence 
\begin{align}\label{eq:part4int2}
\abs{Y(y,t)} \geq \Real(Y(y,t)) \geq \Real(y) \geq \frac{\abs{y}}{\sqrt{2}} \gtrsim \abs{y}^{1+\ep} 
\end{align}
and thus $\abs{X(x,t)} \gtrsim \abs{x}^{1+\ep}$. We define $T^* = T_5$ and the proof is complete. 
\end{enumerate}
\end{proof}
\begin{rmk}
As stated in the introduction, the assumption of $\supp\brac{\w_0} \subset \Omegaplusbar$ can be relaxed to $ \supp(\w_0) \subset \Omega_\beta = \cbrac{re^{i\theta} \in \Csp \suchthat r\geq 0 \tx{ and } 0 \leq \theta \leq \beta(\nu)\nu\pi}$ for some $\beta(\nu)>\half$. To do this, part (3) and part (4) of the above proposition need to be modified. First observe that part (4) needs very little change. Indeed the new statement would be that for $x \in \Omega_\beta\cap B_R(0)$ we have $\abs{X(x,t)} \gtrsim_{\beta} \abs{x}^{1+\ep}$. To prove this we follow the same proof as above and see that as $y = x^\frac{1}{\nu} \in \cbrac{re^{i\theta} \suchthat 0<\theta <\beta \pi}$, therefore there exists $c_\beta >0$ such that either $ c_\beta \abs{\Real(y)} \leq \Imag(y)$ or $\Imag(y) \leq \Real(y)$. In the first case, from \eqref{eq:part4int1} we see that $\abs{Y(y,t)} \gtrsim_\beta \abs{y}^{1+ \ep}$ and in the second case \eqref{eq:part4int2} gives $\abs{Y(y,t)} \gtrsim \abs{y}^{1+ \ep}$. To prove the analog of part (3), define $\tilde{\Psi}:\Omega \to \Hsp$ by $\tilde{\Psi}(z) = z^{\frac{1}{2\beta\nu}}$. For this map we have $\tilde{\Psi}(\Omega_\beta) = \Hspplus$. Let $w = \tilde{\Psi}(x)$ and define a new flow $W$ in $\tilde{\Psi}(\Omega)$ by $W(w,t) = X(x,t)^{\frac{1}{2\beta\nu}}$. Now following the same argument for $W$ instead of $Y$ gives us the new estimate. Indeed by choosing an $\ep$ small enough (depending on $b_0$ and $\beta$) and letting $R^{**} = R^{\frac{1}{2\beta \nu}}$, we see that for all $W(w,t) \in \Hspplus\cap B_{\sqrt{2}R^{**}}(0)$ we have
\begin{align*}
\frac{\diff}{\diff t}\Real(W(w,t)) & = \frac{1}{2\beta \nu^2}\abs{W(w,t)}^{4\beta(1-\nu)}\Real\brac{b(X(x,t),t) W(w,t)^{1-2\beta}} \\
& \gtrsim_\beta (b_0 - \ep)\abs{W(w,t)}^{4\beta(1-\nu) + 1 - 2\beta}.
\end{align*}
Hence for $w \in \Hspplus\cap B_{R^{**}}(0)$, we see that $W(w,t) \in \Hspplus$ for $t \in [0,T_3^*]$, for a suitably modified $T_3^*>0$. Following a similar argument as in the proposition, we also get a quantitive lower bound for $\abs{X(x,t)}$ for $x \in \Omega_\beta\cap B_R(0)$. 

\end{rmk}

We now prove that around the corner the flow moves to the right for \emph{all time} and particles in $\Omegaplus$ cannot come very close to the origin. We need this to prove uniqueness for all time in \thmref{thm:main} and not just for a short time. The following lemma is the analog of proving that $x_1(t), x_2(t) \geq c$ for $t\geq T$ in step (4) of the proof of uniqueness of the ODE \eqref{eq:sampleODE} in the introduction. Proving this lemma would be immediate by a continuity argument if we knew that $X:\Omega \times [0,\infty) \to \Omega$ extends continuously to $X:\Omegabar\times [0,\infty) \to \Omegabar$. We suspect that this is true but do not have an argument for it. As we do not know this property, the proof of the following lemma becomes a little more involved. 

\begin{lemma}\label{lem:postime}
Let $(u,\w)$ be a Yudovich weak solution in the domain $\Omega$ in the time interval $[0,\infty)$ with initial vorticity $\w_0 \in \Lone(\Omega)\cap\Linfty(\Omega)$ satisfying $\w_0 \geq 0$ and $\w_0 \not\equiv 0$. Let $X:\Omega\times[0,\infty) \to \Omega$ be the flow map of the solution and let $0<T_1<T_2$. Then there exists $c>0$ such that $\abs{X(x,t)} \geq c>0$ for all $x\in \Omegaplus$ and $t\in[T_1,T_2]$.
\begin{proof}
Put $\ep = \frac{1}{2}\min\cbrac{b_0,1}>0$ and let $R,T>0$ be as given by \propref{prop:flow}, from which in particular we get $0<R<1/10$. Let $R^* = R^{\frac{1}{\nu}} < 1$ and let $T_0 = \min\cbrac{T,T_1}>0$. Define $\delta>0$ as
\begin{align*}
\delta = \sqbrac{\frac{(2\nu-1)(b_0 - \ep)T_0}{\nu^2}}^{\frac{1}{2\nu-1}} > 0.
\end{align*}
Hence from \eqref{eq:ReYest} we see that for all $y\in \Hspplus\cap B_{R^*}(0)$ and we have
\begin{align*}
\Real\cbrac{Y(y,T_0)} \geq \delta  > 0.
\end{align*}

Now observe that $\btil:\Hspbar\times[0,T_2] \to \Csp$ is continuous from \lemref{lem:transport}. Also observe that for $y\in \Rsp$ we have from \eqref{eq:btilnew}
\begin{align*}
\btil(y,t) & =  \brac*[\Big]{\frac{i\nu^2}{2\pi}} \int_{\Hsp}\sqbrac{\frac{1}{\Real(\ybar) - \sbar} - \frac{1}{\Real(\ybar) - s}}\wtil(s,t) \abs{s}^{2\nu -2} \diff s \\
& = \frac{\nu^2}{\pi}\int_\Hsp \frac{\Imag(s)}{\abs{\Real(y) - s}^2} \wtil(s,t)\abs{s}^{2\nu -2} \diff s \\
& >0.
\end{align*}
Hence there exists $R_1> 1$ and $0<\delta_1< R^*/2$ so that $\Real(\btil(y,t)) \geq 0$ for all $y\in [-R_1,R_1]\times [0,\delta_1]$ and $t\in [0,T_2]$. Now from \lemref{lem:dist} we see that there exists  $0<r < \min\cbrac{\delta, \delta_1}$ such that for all $y\in \Hsp\cap B_r(0)$ and $t_1, t_2\in[0,T_2]$, we have $Y(Y^{-1}(y,t_1), t_2) \in [-R_1,R_1]\times (0,\delta_1]$. 

\begin{figure}[ht]
\centering
 \begin{tikzpicture}
\draw [thick] (-4.5,0) -- (9.5,0) (2.5, -1) -- (2.5,3) (-4,0.0) -- (-4,0.5) -- (9,0.5) -- (9,0);
\draw (2.5 + 2,0)  arc[radius = 2, start angle= 0, end angle= 180];
\draw (2.5 + 0.3,0)  arc[radius = 0.3, start angle= 0, end angle= 180];
\node at (9,-0.35) {$R_1$};
\node at (2.5+0.35,-0.25) {$r$};
\node at (2.5+2,-0.25) {$R^*$};
\draw [thick,  <->] (9.2,0) to (9.2,0.5) ;
\node at (9.6,0.25) {$\delta_1$};
\node at (7,2) {$r< \delta$};
\node at (7,2.5) {$2r<2\delta_1<R^*<1<R_1$};
\end{tikzpicture}
\caption{Particle flow around the boundary}\label{fig:main}
\end{figure}
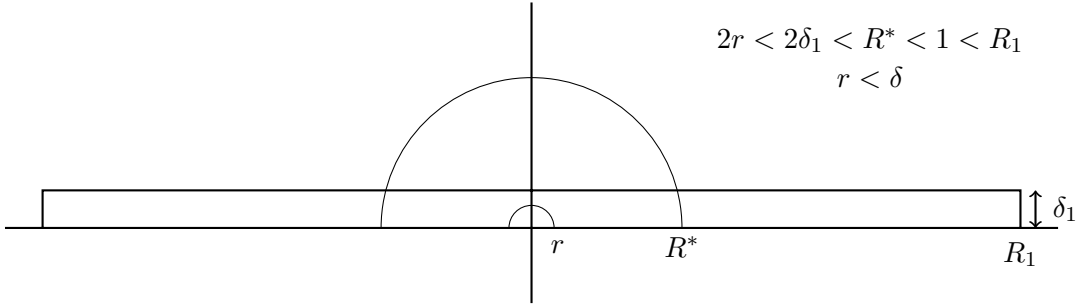

We now claim that for all $y\in \Hspplus$ and $t\in [T_0,T_2]$ we have $\abs{Y(y,t)} \geq r$. We show this via contradiction. Suppose $y_0\in \Hspplus$ and $t_0\in [T_0,T_2]$ be such that $Y(y_0,t_0) \in B_r(0)$. Then by definition of $r$ we have that $Y(y_0,t) \in [-R_1,R_1]\times (0,\delta_1]$ for all $t\in [0,T_2]$.  Hence from \eqref{eq:Ynew} we have
\begin{align*}
\frac{\diff \Real\cbrac{Y(y_0,t)}}{\diff t} = \frac{1}{\nu^2}\Real(\btil(Y(y_0,t),t))\abs{Y(y_0,t)}^{2-2\nu}  \geq 0.
\end{align*}

Now if $y_0 \in \Hspplus\cap B_{R^*}(0)$ then we see that $\Real\cbrac{Y(y_0,T_0)}\geq \delta>r$. Hence by the above estimate we have $\Real\cbrac{Y(y_0,t_0)} \geq \Real\cbrac{Y(y_0,T_0)} >r$ and hence $\abs{Y(y_0,t_0)} > r$, which is a contradiction.

 On the other hand if $y_0 \notin \Hspplus\cap B_{R^*}(0)$ but satisfies $y_0\in [-R_1,R_1]\times (0,\delta_1]$ and $y_0 \in \Hspplus$, then from the fact that $\delta_1 < R^*/2$ implies $\Real(y_0) \geq R^*/2>r$. Hence we similarly obtain $\Real\cbrac{Y(y_0,t_0)}>r$ and hence $\abs{Y(y_0,t_0)} > r$, which is a contradiction. The lemma now follows by setting $c = r^\nu$.
\end{proof}
\end{lemma}

\section{Energy Estimate}\label{sec:energy}

In this section we will ignore the dependence of constants on $\nu$ and $\norm[\Lone\cap\Linfty]{\w_0}$. So we write $a\lesssim b$ instead of $a \lesssim_{\nu,\norm[\Lone\cap\Linfty]{\w_0}} b$.

We now consider two Yudovich weak solutions $(u_1,\w_1)$ and $(u_2,\w_2)$ in $\Omega$ in the time interval $[0,\infty)$ with the same initial vorticity $\w_0$ satisfying \eqref{assump}. Let  $X_1,X_2:\Omega\times[0,\infty) \to \Omega$ be the corresponding flows of the solutions. Let $b_1,b_2:\Omega\times [0,\infty) \to \Csp$ be the corresponding functions from \eqref{eq:Xnew},\eqref{eq:bnew} so that
\begin{align*}
\frac{\diff X_1(x,t)}{\diff t} = \frac{1}{\nu}b_1(X_1(x,t),t)\Xbar_1(x,t)^{\brac{\frac{1}{\nu} -1}} \qq \frac{\diff X_2(x,t)}{\diff t} = \frac{1}{\nu}b_2(X_2(x,t),t)\Xbar_2(x,t)^{\brac{\frac{1}{\nu} -1}}.
\end{align*}
Consider the energy
\begin{align*}
E_1(t) = \int_\Omega \abs{X_1(x,t) - X_2(x,t)} \abs{\w_0(x)} \diff x.
\end{align*}
We first prove that this energy cannot grow too fast near $t=0$. Note that if the domain is $C^{1,1}$ or if the corner has an angle of $\nu\pi$ with $0<\nu\leq \half$, then the energy  $E_1(t)$ is sufficient to prove uniqueness and one can directly show that $E_1(t) = 0$ for all $t\geq 0$ (See the uniqueness proof in Sec 2.3 of \cite{MaPu94} or simply follow the proof of the proposition below). What we show below is that even if the energy $E_1(t)$ is not sufficient to prove uniqueness in our case, we can still gain some useful information out of it. 

\begin{prop}\label{prop:Eone}
Let $(u_1,\w_1)$ and $(u_2,\w_2)$ be two Yudovich weak solutions in $\Omega$ in the time interval $[0,\infty)$ with the same initial vorticity $\w_0 \in \Lone(\Omega)\cap\Linfty(\Omega)$ satisfying $ \supp\brac{\w_0} \subset \Omegaplusbar$ and $b_0 >0$ where $b_0$ is defined in \eqref{eq:bzero}. Let $X_1,X_2:\Omega\times[0,\infty) \to \Omega$ be the corresponding flows of the solutions.  Then for any $\alpha>0$ satisfying $1<\alpha < \frac{2\nu}{2\nu-1}$ there exists constants $C,T>0$ such that for all $t\in[0,T]$ we have
\begin{align*}
E_1(t) \leq C t^\alpha.
\end{align*}
\end{prop}
\begin{proof}
We first define constants depending on $\alpha,\nu$ and on $b_0$. Define
\begin{align}\label{eq:constants}
p =   \frac{\alpha}{\alpha - 1} > 1  \quad \tx{ and }\quad  \ep = \half\min\cbrac{\frac{2\nu - \alpha(2\nu -1)}{2\alpha(2\nu -1)}, \frac{b_0(2\nu-1)}{(3-2\nu)}, 1 } >0.
\end{align}
Hence $p>2\nu > 1$ and  $0<\frac{2}{p} < \frac{1 + \ep(1-2\nu)}{\nu}$. Also $0<\frac{b_0 + \ep}{b_0 - \ep} < \frac{1}{2(1-\nu)}$ and $0<\ep<\min\cbrac{b_0, 1}$. (Here the estimate $p>2\nu$ comes due to the restriction $\alpha < \frac{2\nu}{2\nu -1}$ in the proposition). We will use these inequalities in the upcoming computation. 

Observe that as the velocity is bounded by \lemref{lem:velocity}, we see that for all $x \in \Omega$ and $t\in [0,\infty)$ we have $\abs{X_1(x,t) - X_2(x,t)}  \lesssim t$. Hence there exists $T^*_1>0$ so that for all $x \in \Omega$ and $t\in [0,T^*_1]$ we have 
\begin{align*}
\abs{X_1(x,t) - X_2(x,t)} \leq 1/10.
\end{align*}
Again using the fact that the velocity is bounded we get
\begin{align*}
 \frac{\diff E_1(t)}{\diff t} \leq \int_\Omega \abs{\frac{\diff X_1(x,t)}{\diff t} - \frac{ \diff X_2(x,t)}{\diff t}} \abs{\w_0(x)}\diff x  \lesssim 1
\end{align*}
and hence $E_1(t)  \lesssim t$. Thus there exists $T^*_2>0$ such that for all $t\in[0,T^*_2]$ we have 
\begin{align*}
E_1(t) \leq \frac{1}{10}\min\cbrac{1, \norm[1]{\w_0}}.
\end{align*}
Now using the $\ep$ from \eqref{eq:constants} in \propref{prop:flow} for the flows $X_1(\cdot,t), X_2(\cdot,t)$ we get constants $R_1,R_2$ and $T_1,T_2$ so that \propref{prop:flow} is satisfied. Let 
\begin{align}\label{eq:RT}
R = \min\cbrac{R_1,R_2} >0, \qq T = \min\cbrac{T^*_1, T^*_2, T_1,T_2}>0.
\end{align}

Now for $t\in[0,\infty)$ we have
\begin{align}\label{eq:onetwo}
\begin{split}
 \frac{\diff E_1(t)}{\diff t} & \leq \int_\Omega \abs{\frac{\diff X_1(x,t)}{\diff t} - \frac{ \diff X_2(x,t)}{\diff t}} \abs{\w_0(x)} \diff x \\
& \leq \frac{1}{\nu}\int_\Omega \abs{b_1(X_1(x,t),t)\Xbar_1(x,t)^{\frac{1}{\nu}-1} - b_2(X_2(x,t),t)\Xbar_2(x,t)^{\frac{1}{\nu}-1}} \abs{\w_0(x)} \diff x \\
& \leq  \frac{1}{\nu}\int_\Omega \abs{b_1(X_1(x,t),t)}\abs{\Xbar_1(x,t)^{\frac{1}{\nu}-1} - \Xbar_2(x,t)^{\frac{1}{\nu}-1}} \abs{\w_0(x)}\diff x \\
& \quad + \frac{1}{\nu}\int_\Omega \abs{b_1(X_1(x,t),t) - b_2(X_2(x,t),t)}\abs{X_2(x,t)}^{\frac{1}{\nu} - 1}\abs{\w_0(x)} \diff x  \\
& = \rom{1} + \rom{2}.
\end{split}
\end{align} 

\smallskip
\noindent$\mathbf{Controlling \enspace \rom{1}:}$
We first control the first term $\rom{1}$ in \eqref{eq:onetwo}.  Using \lemref{lem:transport},\lemref{lem:powers} and \propref{prop:flow} we have for all $t\in [0,T]$ 
\begingroup
\allowdisplaybreaks
\begin{align*}
\rom{1} & = \frac{1}{\nu}\int_{\Omegaplus} \abs{b_1(X_1(x,t),t)}\abs{\Xbar_1(x,t)^{\frac{1}{\nu}-1} - \Xbar_2(x,t)^{\frac{1}{\nu}-1}} \abs{\w_0(x)}\diff x \\
& \lesssim \int_{\Omegaplus\cap B_R(0)} \abs{\Xbar_1(x,t)^{\frac{1}{\nu}-1} - \Xbar_2(x,t)^{\frac{1}{\nu}-1}} \abs{\w_0(x)}\diff x \\*
& \quad +  \int_{\Omegaplus\cap B_R^c(0)} \abs{\Xbar_1(x,t)^{\frac{1}{\nu}-1} - \Xbar_2(x,t)^{\frac{1}{\nu}-1}} \abs{\w_0(x)}\diff x \\
& \lesssim  \int_{\Omegaplus\cap B_R(0)} \abs{x}^{(\frac{1}{\nu} - 2)(1+\ep)}  \abs{X_1(x,t) - X_2(x,t)}\abs{\w_0(x)} \diff x \\*
& \quad +  \int_{\Omegaplus\cap B_R(0)^c} \abs{R}^{(\frac{1}{\nu}-2)}  \abs{X_1(x,t) - X_2(x,t)}\abs{\w_0(x)} \diff x\\
& \lesssim_{R} \norm*[\big][L^p(\Omegaplus\cap B_R(0))]{\abs{X_1(x,t) - X_2(x,t)}\abs{\w_0(x)}}\norm[L^q(B_R(0))]{\abs{x}^{(\frac{1}{\nu} - 2)(1+\ep)}} \\
& \quad + E_1(t),
\end{align*}
\endgroup
where $\frac{1}{p} + \frac{1}{q} = 1$. For $\abs{x}^{(\frac{1}{\nu} - 2)(1+\ep)} \in L^q(B_R(0))$ we need
\begin{align*}
& \qq \qq \brac*[\Big]{\frac{1}{\nu} - 2}(1+\ep)q  > -2 \\
&\iff  \qq  \frac{1}{\nu} -2 + \frac{\ep(1-2\nu)}{\nu}  > -2\brac{1 - \frac{1}{p}} \\
&\iff \qq \frac{1 + \ep(1-2\nu)}{\nu}  > \frac{2}{p},
\end{align*}
which is satisfied by the choice of $p $ and $\ep$ given in \eqref{eq:constants}. Now as  $\abs{X_1(x,t) - X_2(x,t)} \leq 1/10$, $E_1(t)\leq 1/10$ and $\norm[\infty]{\w_0}<\infty$, we see that for all $t\in[0,T]$ we have
\begin{align}\label{eq:estimateone}
\rom{1}  \lesssim_{\alpha, b_0, R} E_1(t)^{\frac{1}{p}} + E_1(t)  \lesssim_{\alpha, b_0, R} E_1(t)^{\frac{1}{p}}.
\end{align}

\smallskip
\noindent $\mathbf{Controlling \enspace \rom{2}:}$ We now control the second term $\rom{2}$ in \eqref{eq:onetwo}. From the definition of $b$ in \eqref{eq:bnew} we see that
\begingroup
\allowdisplaybreaks
\begin{align*}
&  \abs{b_1(X_1(x,t),t) - b_2(X_2(x,t),t)} \\
 & \lesssim \int_\Omega \abs{\frac{1}{\Xonebar(x,t)^{\frac{1}{\nu}} - \sbar^{\frac{1}{\nu}}} - \frac{1}{\Xtwobar(x,t)^{\frac{1}{\nu}} - \sbar^{\frac{1}{\nu}}}}\abs{\w_1(s,t)} \diff s \\
& \quad + \int_\Omega \abs{\frac{1}{\Xonebar(x,t)^{\frac{1}{\nu}} - s^{\frac{1}{\nu}}} - \frac{1}{\Xtwobar(x,t)^{\frac{1}{\nu}} - s^{\frac{1}{\nu}}}}\abs{\w_1(s,t)} \diff s \\
& \quad + \abs{\int_\Omega \frac{1}{\Xtwobar(x,t)^{\frac{1}{\nu}} - \sbar^{\frac{1}{\nu}}}(\w_1(s,t) - \w_2(s,t))\diff s } \\
& \quad + \abs{\int_\Omega \frac{1}{\Xtwobar(x,t)^{\frac{1}{\nu}} - s^{\frac{1}{\nu}}}(\w_1(s,t) - \w_2(s,t))\diff s }  \\
& = \rom{2}_1 + \rom{2}_2 + \rom{2}_3 + \rom{2}_4.
\end{align*}
\endgroup
Now by \lemref{lem:powers}  we see that
\begin{align*}
&\rom{2}_1 + \rom{2}_2\\
& \lesssim \int_\Omega \frac{\abs{X_1(x,t)^{\frac{1}{\nu}} - X_2(x,t)^{\frac{1}{\nu}} }}{\abs{X_1(x,t)^{\frac{1}{\nu}} - s^{\frac{1}{\nu}} }\abs{X_2(x,t)^{\frac{1}{\nu}} - s^{\frac{1}{\nu}} } }\abs{\w_1(s,t)} \diff s \\
& \lesssim \int_{\Omega} \frac{\abs{X_1(x,t) - X_2(x,t)}\max\cbrac{\abs{X_1(x,t)}^{\frac{1}{\nu} -1}, \abs{X_2(x,t)}^{\frac{1}{\nu} - 1}} \abs{\w_1(s,t)} \diff s}{\abs{X_1(x,t) - s}\abs{X_2(x,t) - s}\max\cbrac{\abs{X_1(x,t)}^{\frac{1}{\nu} -1}, \abs{s}^{\frac{1}{\nu} - 1}}\max\cbrac{\abs{X_2(x,t)}^{\frac{1}{\nu} -1}, \abs{s}^{\frac{1}{\nu} - 1}} } \\
& \lesssim \int_\Omega \frac{\abs{X_1(x,t) - X_2(x,t)}}{\abs{X_1(x,t) - s}\abs{X_2(x,t) - s}\abs{s}^{\frac{1}{\nu} -1}} \abs{\w_1(s,t)} \diff s.
\end{align*}
If $\zone = X_1(x,t)$, $\ztwo = X_2(x,t)$ and $f = \abs{\w_1(\cdot,t)}\one_{\Omega}$ then we see that the above integral equals $ \abs{z_1 - z_2}I((0,\frac{1}{\nu}-1), (z_1,1), (z_2,1): (f,0,\infty))$ (note that $I$ is defined in \eqref{def:I}). Hence by the first estimate of \lemref{lem:Ithree} we get
\begin{align*}
\rom{2}_1 + \rom{2}_2  \lesssim  \phi\brac{\abs{X_1(x,t) - X_2(x,t)}} \min\cbrac{\abs{X_1(x,t)}^{1 - \frac{1}{\nu}}, \abs{X_2(x,t)}^{1 - \frac{1}{\nu}}}.
\end{align*}
Now let us concentrate on $ \rom{2}_3$ and $ \rom{2}_4$. As $\w_1(X_1(s,t),t) = \w_0(s)$ from \lemref{lem:transport} and as the mapping $X_1(\cdot,t)$ is measure preserving (and its inverse as well), we see from the change of variable $s \mapsto X_1(s,t)$ on the first term of $ \rom{2}_3$
\begin{align*}
\int_\Omega \frac{1}{\Xtwobar(x,t)^{\frac{1}{\nu}} - \sbar^{\frac{1}{\nu}}} \w_1(s,t) \diff s = \int_\Omega \frac{1}{\Xtwobar(x,t)^{\frac{1}{\nu}} - \Xonebar(s,t)^{\frac{1}{\nu}}} \w_0(s) \diff s.
\end{align*}
By doing a similar change of variable for the second term of $ \rom{2}_3$ as well, we see that
\begin{align*}
 \rom{2}_3 = \abs{\int_\Omega \frac{1}{\Xtwobar(x,t)^{\frac{1}{\nu}} - \Xonebar(s,t)^{\frac{1}{\nu}}} \w_0(s) \diff s - \int_\Omega \frac{1}{\Xtwobar(x,t)^{\frac{1}{\nu}} - \Xtwobar(s,t)^{\frac{1}{\nu}}} \w_0(s) \diff s}.
\end{align*}
Now by a similar computation for $ \rom{2}_4$ and by using  \lemref{lem:powers} we now see that
\begin{align*}
\rom{2}_3 + \rom{2}_4 &  \lesssim \int_\Omega \frac{\abs{X_1(s,t)^{\frac{1}{\nu}} - X_2(s,t)^{\frac{1}{\nu}} }}{\abs{X_2(x,t)^{\frac{1}{\nu}} - X_1(s,t)^{\frac{1}{\nu}} }\abs{X_2(x,t)^{\frac{1}{\nu}} - X_2(s,t)^{\frac{1}{\nu}} } }\abs{\w_0(s)} \diff s \\
& \lesssim \int_\Omega \frac{\abs{X_1(s,t) - X_2(s,t)}}{\abs{X_2(x,t) - X_1(s,t)}\abs{X_2(x,t) - X_2(s,t)}\abs{X_2(x,t)}^{\frac{1}{\nu} -1}} \abs{\w_0(s)} \diff s. 
\end{align*}
Combining all these estimates we get from \eqref{eq:onetwo} that,
\begin{align*}
 \rom{2} & = \frac{1}{\nu}\int_\Omega \abs{b_1(X_1(x,t),t) - b_2(X_2(x,t),t)}\abs{X_2(x,t)}^{\frac{1}{\nu} - 1}\abs{\w_0(x)} \diff x \\
& \lesssim  \int_\Omega \abs{\rom{2}_1 + \rom{2}_2 + \rom{2}_3 + \rom{2}_4}\abs{X_2(x,t)}^{\frac{1}{\nu} - 1}\abs{\w_0(x)} \diff x \\
& \lesssim \int_\Omega \phi\brac{\abs{X_1(x,t) - X_2(x,t)}} \abs{\w_0(x)} \diff x \\
& \quad + \int_\Omega \cbrac{ \int_\Omega \frac{\abs{X_1(s,t) - X_2(s,t)}}{\abs{X_2(x,t) - X_1(s,t)}\abs{X_2(x,t) - X_2(s,t)}} \abs{\w_0(s)} \diff s } \abs{\w_0(x)} \diff x.
\end{align*}
Now by using Fubini and using the change of variable $X_2(x,t) \mapsto x$ while observing that this is measure preserving, we obtain 
\begin{align*}
 \rom{2} & \lesssim \int_\Omega \phi\brac{\abs{X_1(x,t) - X_2(x,t)}} \abs{\w_0(x)} \diff x \\
& \quad + \int_\Omega \cbrac{\int_\Omega \frac{\abs{X_1(s,t) - X_2(s,t)}}{\abs{x - X_1(s,t)}\abs{x - X_2(s,t)}} \abs{\w_0(X_2^{-1}(x,t))} \diff x } \abs{\w_0(s)}\diff s.
\end{align*}
Now if $\zone = X_1(s,t)$, $\ztwo = X_2(s,t)$ and $f = \abs{\w_0(X_2^{-1}(\cdot,t))} \one_{\Omega}$ then we see that the inner integral in the second term equals $\abs{z_1 - z_2}I((z_1,1),(z_2,1):(f,0,\infty))$ (as defined in \eqref{def:I}). Hence using \lemref{lem:Itwo} and the fact that $\norm[\Lone\cap\Linfty]{\w_0(X_2^{-1}(\cdot,t))} = \norm[\Lone\cap\Linfty]{\w_0}$  we get
\begin{align*}
\rom{2} \lesssim \int_\Omega \phi\brac{\abs{X_1(x,t) - X_2(x,t)}} \abs{\w_0(x)} \diff x.
\end{align*}
As $\abs{X_1(x,t) - X_2(x,t)} \leq 1/10$ for all $x\in \Omega$ and $t \in [0,T]$ and as $\phi(x)$ is a concave function in $[0,1/10]$, we obtain from Jensen's inequality
\begin{align*}
\rom{2} \lesssim \norm[1]{\w_0}\phi\brac{ \frac{E_1(t)}{\norm[1]{\w_0}}}.
\end{align*}
Now as $E_1(t)\leq (1/10)\min\cbrac{1, \norm[1]{\w_0}}$ in the interval $[0,T]$, we use the formula of $\phi$ from \eqref{def:phi} to see that   
\begin{align}\label{eq:estimatetwo}
\rom{2} \lesssim E_1(t)\brac{ - \ln(E_1(t)) + \ln(\norm[1]{\w_0})}   \lesssim \phi(E_1(t)).
\end{align}
We can now use the estimates \eqref{eq:estimateone} and \eqref{eq:estimatetwo} in the equation \eqref{eq:onetwo} to see that
\begin{align*}
 \frac{\diff E_1(t)}{\diff t}   \lesssim_{\alpha, b_0, R} E_1(t)^{\frac{1}{p}}.
\end{align*}
Hence by integrating we see that there exists a $C_1>0$ such that
\begin{align*}
E_1(t)^{1- \frac{1}{p}} \leq C_1t.
\end{align*}
We get the result by  observing that $\alpha = \frac{p}{p-1}$. 
\end{proof}

We are now ready to prove our main result. 

\begin{proof}[Proof of \thmref{thm:main}]
Let $(u_1,\w_1)$ and $(u_2,\w_2)$ be two Yudovich weak solutions in $\Omega$ in the time interval $[0,\infty)$ with the same initial vorticity $\w_0$ satisfying \eqref{assump}. If $\w_0 \equiv 0$ then the result is obviously true and so we can assume that $\w_0 \not\equiv 0$, which in turn implies that $b_0 > 0$. Let $X_1,X_2:\Omega\times[0,\infty) \to \Omega$ be the corresponding flows of the solutions.

Let $ \alpha = \frac{1}{2\nu -1} $ so that $\alpha$ satisfies the conditions of \propref{prop:Eone}. Let
\begin{align*}
E(t) =  t^{-\alpha}E_1(t) =  t^{-\alpha} \int_\Omega \abs{X_1(x,t) - X_2(x,t)} \abs{\w_0(x)} \diff x.
\end{align*}
From \propref{prop:Eone} we see that $\lim_{t \to 0^+} E(t) = 0$. We now use this energy to prove uniqueness in a time interval $[0,T^*]$ for some $T^*>0$.

Let $\ep, R,T>0$ be as defined in the proof of \propref{prop:Eone} for the value of $\alpha = \frac{1}{2\nu -1}$ given in \eqref{eq:constants} and \eqref{eq:RT}. For all $t\in[0,T]$ we see that
\begin{align*}
\frac{\diff E(t)}{\diff t} = t^{-\alpha} \cbrac{\brac{\frac{-\alpha}{t}}E_1(t) + \frac{\diff E_1(t)}{\diff t} }.
\end{align*}
From the estimates \eqref{eq:onetwo}, \eqref{eq:estimatetwo} and the computation for \eqref{eq:estimateone} obtained in the proof of \propref{prop:Eone} we get
\begingroup
\allowdisplaybreaks
\begin{align*}
& \brac{\frac{-\alpha}{t}}E_1(t) + \frac{\diff E_1(t)}{\diff t} \\
& \leq \brac{\frac{-\alpha}{t}}E_1(t) +  \rom{1} + \rom{2}  \\
& \leq \cbrac{\brac{\frac{-\alpha}{t}}E_1(t) + \frac{1}{\nu}\int_{\Omegaplus\cap B_R(0)} \abs{b_1(X_1(x,t),t)}\abs{\Xbar_1(x,t)^{\frac{1}{\nu}-1} - \Xbar_2(x,t)^{\frac{1}{\nu}-1}} \abs{\w_0(x)}\diff x} \\*
& \quad + \frac{1}{\nu}\int_{\Omegaplus\cap B_R^c(0)} \abs{b_1(X_1(x,t),t)}\abs{\Xbar_1(x,t)^{\frac{1}{\nu}-1} - \Xbar_2(x,t)^{\frac{1}{\nu}-1}} \abs{\w_0(x)}\diff x + \rom{2} \\
& \lesssim_{b_0, R}   \cbrac{\brac{\frac{-\alpha}{t}}E_1(t) + \frac{1}{\nu}\int_{\Omegaplus\cap B_R(0)} \abs{b_1(X_1(x,t),t)}\abs{X_1(x,t)^{\frac{1}{\nu}-1} - X_2(x,t)^{\frac{1}{\nu}-1}} \abs{\w_0(x)}\diff x} \\*
& \quad + E_1(t) + \phi(E_1(t)).
\end{align*}
\endgroup
Now for any $z_1,z_2 \in \Hspplus$ we see that  $\min_{z \in [z_1,z_2]} \abs{z} \geq \frac{1}{\sqrt{2}}\min\cbrac{\abs{z_1}, \abs{z_2}}$. From \propref{prop:flow} part (3) we see that for $x \in \Omegaplus\cap B_R(0)$, both $X_1(x,t), X_2(x,t) \in \Omegaplus \subset \Hspplus$. Hence from \propref{prop:flow} part (3) we have
\begingroup
\allowdisplaybreaks
\begin{align*}
& \abs{X_1(x,t)^{\frac{1}{\nu}-1} - X_2(x,t)^{\frac{1}{\nu}-1}} \\
& \leq \brac{\frac{1}{\nu} - 1}\abs{X_1(x,t) - X_2(x,t)}\max_{z \in [X_1(x,t), X_2(x,t)]} \abs{z}^{\frac{1}{\nu} - 2} \\
& \leq  (\sqrt{2})^{2 - \frac{1}{\nu}}\brac{\frac{1}{\nu} - 1}\abs{X_1(x,t) - X_2(x,t)}\max\cbrac{\abs{X_1(x,t)}^{\frac{1}{\nu} - 2}, \abs{X_2(x,t)}^{\frac{1}{\nu} - 2}} \\
& \leq 2\brac{\frac{1}{\nu} - 1}\abs{X_1(x,t) - X_2(x,t)}\sqbrac{\frac{(2\nu-1)(b_0 - \ep)t}{\nu^2}}^{-1} \\
& \leq \frac{2(1-\nu)\nu}{(2\nu -1)(b_0 - \ep) t}\abs{X_1(x,t) - X_2(x,t)}.
\end{align*}
\endgroup
Hence from \propref{prop:flow} part (1) and (2) we have
\begin{align*}
& \brac{\frac{-\alpha}{t}}E_1(t) + \frac{1}{\nu}\int_{\Omegaplus\cap B_R(0)} \abs{b_1(X_1(x,t),t)}\abs{X_1(x,t)^{\frac{1}{\nu}-1} - X_2(x,t)^{\frac{1}{\nu}-1}} \abs{\w_0(x)}\diff x \\
& \leq \brac{\frac{-\alpha}{t}}\int_{\Omegaplus\cap B_R(0)} \abs{X_1(x,t) - X_2(x,t)} \abs{\w_0(x)} \diff x \\
& \quad + \frac{2(1-\nu)(b_0 + \ep)}{(2\nu -1)(b_0 - \ep) t}\int_{\Omegaplus\cap B_R(0)} \abs{X_1(x,t) - X_2(x,t)} \abs{\w_0(x)} \diff x \\
& \leq \brac{-\alpha + \frac{2(1-\nu)(b_0 + \ep)}{(2\nu -1)(b_0 - \ep) }}\frac{1}{t}\int_{\Omegaplus\cap B_R(0)} \abs{X_1(x,t) - X_2(x,t)} \abs{\w_0(x)} \diff x.
\end{align*}
This term is non-positive as $\alpha = \frac{1}{2\nu -1}$ and by the choice of $\ep$ from \eqref{eq:constants} we have $0< \frac{b_0 + \ep}{b_0 - \ep} < \frac{1}{2(1-\nu)}$. Hence for all $t\in [0,T]$ we have
\begin{align*}
\frac{\diff E(t)}{\diff t} \lesssim_{b_0, R} t^{-\alpha}\phi(E_1(t)).
\end{align*}
Now as $0\leq E_1(t)\leq 1/10$ in $t\in [0,T]$ we have
\begin{align*}
\frac{\diff E(t)}{\diff t} \lesssim_{b_0, R} -t^{-\alpha}E_1(t)\ln(E_1(t)).
\end{align*}
Now let $\beta>0$ be such that $0< \alpha < \beta < \frac{2\nu}{2\nu -1}$. From \propref{prop:Eone} we see that there exists a constant $C_2>0$ such that $E_1(t) \leq C_2t^{\beta}$ for all $t \in [0,T]$. Hence $E(t) \leq C_2t^{\beta - \alpha}$. Let $T^* = \min\cbrac{T, (10C_2)^{-\frac{1}{(\beta - \alpha)}}} >0$ and so for all $t\in [0,T^*]$ we have $0\leq E(t)\leq C_2t^{\beta - \alpha} \leq 1/10$, and thus for all $t\in [0,T^*]$ we have
\begin{align*}
\frac{\diff E(t)}{\diff t} & \lesssim_{b_0, R} -t^{-\alpha}E_1(t)\cbrac{\ln(t^{-\alpha}E_1(t)) + \ln(t^{\alpha}) } \\
&  \lesssim_{ \beta,C_2, b_0, R} \phi(E(t)) - t^{-\alpha}E_1(t)\ln(C_2t^{\beta - \alpha}) \\
&  \lesssim_{\beta,C_2, b_0, R} \phi(E(t)).
\end{align*} 
Hence by \lemref{lem:phi} we have $E(t) = 0$ for $t \in [0,T^*]$. This implies that for a.e. $x \in \supp (\w_0)$ and $t\in[0,T^*]$ we have $X_1(x,t) = X_2(x,t)$. Hence from \lemref{lem:transport} we  see that for a.e. $t\in [0,T^*]$ we have $\supp (\w_1(\cdot,t)) = \supp (\w_2(\cdot,t))$ a.e. and that for a.e. $x\in \supp (\w_1(\cdot,t))$ we have $\w_1(x,t) = \w_2(x,t)$. All in all, we have that $\w_1(x,t) = \w_2(x,t) $ for a.e. $(x,t) \in \Omega \times  [0,T^*]$ and hence   $X_1(x,t) = X_2(x,t)$ a.e. $(x,t) \in \Omega \times [0,T^*] $. 

To complete the proof we will show the uniqueness for any arbitrary large time interval $[0,T']$ where $T' > T^*$. From \lemref{lem:postime} we see that there exists $c>0$ so that  for all $x\in \Omegaplus$ and $t\in [T^*,T']$ we have $\abs{X_i(x,t)}\geq c>0$ for $i=1,2$. Thus by following the proof of \propref{prop:Eone} we see that for all $t\in [T^*,T']$
\begin{align*}
\frac{\diff E_1(t)}{\diff t} \lesssim_{c} \phi(E_1(t)).
\end{align*}
As $E_1(T^*) = 0$, we see that $E_1(t) = 0$ for all $t \in [T^*,T']$ and therefore by similar argument as above we have $\w_1(x,t) = \w_2(x,t) $ for a.e. $(x,t) \in \Omega \times  [0,T']$. Hence proved. 
\end{proof}

\section{Appendix}\label{sec:appendix}

Here we collect some basic estimates we use throughout the paper. 

\begin{lemma}\label{lem:phi}
Let $T,R,c>0$ and let $y:[0,T] \to \Rsp^{+}$ be such that $\abs{y(t)}\leq R$ for all $t\in [0,T]$ and satisfy 
\begin{align*}
\abs{\frac{\diff y}{\diff t}} \leq c\phi(y(t)) \qq\qq y(0) = y_0>0,
\end{align*}
where $\phi$ is given by \eqref{def:phi}. Then 
\begin{align*}
\cbrac{y(0)}^{e^{ct}} \lesssim_R y(t) \lesssim_R \cbrac{y(0)}^{e^{-ct}} \qq \tx{ for all } t\in[0,T].
\end{align*}
\end{lemma}
\begin{proof}
We only prove $y(t) \lesssim_R \cbrac{y(0)}^{e^{-ct}}$ since the other estimate is proved similarly. We have
\begin{align*}
\frac{\diff y}{\diff t} \leq c y \max\cbrac{-\ln (y), 1} \leq c y \cbrac{-\ln(y) + 1 + \ln(R + 1)}.
\end{align*}
Therefore 
\begin{align*}
\frac{\diff \ln(y)}{\diff t}   \leq c \cbrac{-\ln(y) + 1 + \ln(R + 1)}.
\end{align*}
Now multiplying by $e^{ct}$ we obtain
\begin{align*}
\frac{\diff (e^{ct}\ln(y))}{\diff t} \leq e^{ct} c\brac{1 + \ln(R+1)}.
\end{align*}
Integrating the above inequality we get
\begin{align*}
e^{ct}\ln(y(t)) - \ln(y(0)) \leq (e^{ct} - 1)(1 + \ln(R+1)) \leq e^{ct}(1 + \ln(R+1)).
\end{align*}
Hence 
\begin{align*}
\ln(y(t)) \leq e^{-ct}\ln(y(0)) + 1 + \ln(R+1) 
\end{align*}
and so 
\begin{align*}
y(t) \lesssim_R \cbrac{y(0)}^{e^{-ct}}.
\end{align*}
\end{proof}

\begin{lemma}\label{lem:powers}
Let $\nu>0$ and let $a,b \in \Csp$ be non-zero complex numbers satisfying the condition $0\leq\arg(a),\arg(b)<\min\cbrac{\pi,\frac{\pi}{\nu}}$. Then we have
\begin{enumerate}
\item If $0<\nu<1$ then 
\begin{align*}
\quad\qquad \abs{a^\nu - b^\nu} \approx_\nu \abs{a-b}\min\cbrac{\abs{a}^{\nu-1}, \abs{b}^{\nu-1}} \approx_\nu \abs{a-b}\min\cbrac{\abs{a}^{\nu-1}, \abs{b}^{\nu-1}, \abs{a-b}^{\nu-1}}.
\end{align*}
\item If $1<\nu<\infty$ then 
\begin{align*}
\quad\qquad \abs{a^\nu - b^\nu} \approx_\nu \abs{a-b}\max\cbrac{\abs{a}^{\nu-1}, \abs{b}^{\nu-1}} \approx_\nu \abs{a-b}\max\cbrac{\abs{a}^{\nu-1}, \abs{b}^{\nu-1}, \abs{a-b}^{\nu-1}}.
\end{align*}
\end{enumerate}
\end{lemma}
\begin{proof}
We only prove it for $0<\nu<1$ and the proof for $1<\nu<\infty$ is similar. Without loss of generality $\abs{a} \leq \abs{b}$. 
\begin{enumerate}[label=(\alph*)]
\item Case 1: $\abs{a} \leq \frac{\abs{b}}{2}$. We have
\begin{align*}
\abs{a^\nu - b^\nu} \approx_\nu \abs{b}^\nu \approx_\nu \abs{a-b}\abs{b}^{\nu-1}.
\end{align*}
In this case $\abs{b}^{\nu-1} = \min\cbrac{\abs{a}^{\nu-1}, \abs{b}^{\nu-1}} \approx_\nu \min\cbrac{\abs{a}^{\nu-1}, \abs{b}^{\nu-1}, \abs{a-b}^{\nu-1}}$.
\item Case 2: $\frac{\abs{b}}{2} \leq \abs{a} \leq \abs{b}$ and $\abs{\frac{a-b}{b}} \leq \half$. Hence we have
\begin{align*}
\abs{a^\nu - b^\nu} = \abs{b}^\nu\abs{\brac{\frac{a-b}{b} + 1}^\nu - 1}.
\end{align*}
Using the binomial theorem we see that
\begin{align*}
\abs{a^\nu - b^\nu} \approx_\nu \abs{b}^\nu\abs{\frac{a-b}{b}} = \abs{a-b}\abs{b}^{\nu -1}.
\end{align*}
In this case $\abs{b}^{\nu-1} = \min\cbrac{\abs{a}^{\nu-1}, \abs{b}^{\nu-1}} = \min\cbrac{\abs{a}^{\nu-1}, \abs{b}^{\nu-1}, \abs{a-b}^{\nu-1}}$.
\item Case 3: $\frac{\abs{b}}{2} \leq \abs{a} \leq \abs{b}$ and $\half < \abs{\frac{a-b}{b}} < 2$. Observe that
\begin{align*}
\half < \abs{\frac{a}{b} - 1} < 2 \implies \abs{\brac{\frac{a}{b}}^\nu - 1} \approx_\nu 1.
\end{align*}
Hence we have
\begin{align*}
\abs{a^\nu - b^\nu} = \abs{b}^\nu\abs{\brac{\frac{a}{b}}^\nu - 1} \approx_\nu \abs{b}^\nu \approx_\nu \abs{a-b}\abs{b}^{\nu-1}.
\end{align*}
In this case $\abs{b}^{\nu-1} = \min\cbrac{\abs{a}^{\nu-1}, \abs{b}^{\nu-1}} \approx_\nu \min\cbrac{\abs{a}^{\nu-1}, \abs{b}^{\nu-1}, \abs{a-b}^{\nu-1}}$.
\end{enumerate}
\end{proof}

\begin{lemma}\label{lem:Iest}
Let $n\geq 2$ and let $0< R\leq \infty$.  Let $z_1,\cdots, z_n \in \Csp$ be $n$ distinct complex numbers and let $f\in \Linfty(\Csp)$. Let $d_{min} = \min\cbrac{\abs{z_i - z_j}\suchthat 1\leq i,j\leq j, i\neq j} >0$  and let $0\leq r \leq d_{min}/2$.  If $\alpha_1, \cdots, \alpha_n>0$, $d_{min} = \abs{z_1 - z_2}>0$, then for $I$ defined as in \eqref{def:I} we have the following estimate
\begin{align*}
& I((z_1,\alpha_1), \cdots, (z_n,\alpha_n):(f, r,R)) \\
&  \lesssim_{\alpha_1, \cdots, \alpha_n} \norm[\infty]{f}\sum_{i=1}^n \brac{\int_r^{d_{min}/2} \frac{1}{\abs{x}^{(\alpha_i - 1)}} \diff x }\brac{\prod_{j\neq i, 1\leq j\leq n} \abs{z_j - z_i}^{-\alpha_j}} \\
& \quad  + I((z_1,\alpha_1 + \alpha_2), (z_3, \alpha_3), \cdots, (z_n,\alpha_n):(f,d_{min}/2,R)). 
\end{align*}
\end{lemma}

\begin{proof}
Clearly we can assume that $\norm[\infty]{f} >0$. If $r=0$ and $\max \cbrac{\alpha_1, \cdots, \alpha_n} \geq 2$ then the right hand side of the estimate is $\infty$ and there is nothing to prove. Hence we assume that either $r>0$ or that $\max \cbrac{\alpha_1, \cdots, \alpha_n} < 2$. Now  as $0\leq r\leq d_{min}/2$ we have for $1\leq i \leq n$
\begin{align*}
& \int_{B(z_i,r)^c \cap B(z_i,d_{min}/2)} \frac{1}{\abs{s-z_1}^{\alpha_1} \cdots \abs{s-z_n}^{\alpha_n}} \abs{f(s)} \diff s \\
&  \lesssim_{\alpha_1, \cdots, \alpha_n} \norm[\infty]{f} \brac{\prod_{j\neq i, 1\leq j\leq n} \abs{z_j - z_i}^{-\alpha_j}}  \int_{B(z_i,r)^c \cap B(z_i,d_{min}/2)} \frac{1}{\abs{s-z_i}^{\alpha_i} } \diff s \\
&  \lesssim_{\alpha_1, \cdots, \alpha_n} \norm[\infty]{f}\brac{\prod_{j\neq i, 1\leq j\leq n} \abs{z_j - z_i}^{-\alpha_j}} \int_r^{d_{min}/2} \frac{1}{\abs{x}^{(\alpha_i - 1)}} \diff x.
\end{align*}
Summing these up we get
\begin{align*}
& I((z_1,\alpha_1), \cdots, (z_n,\alpha_n):(f,r,R)) \\
&  \lesssim_{\alpha_1, \cdots, \alpha_n} \norm[\infty]{f}\sum_{i=1}^n \brac{\int_r^{d_{min}/2} \frac{1}{\abs{x}^{(\alpha_i - 1)}} \diff x }\brac{\prod_{j\neq i, 1\leq j\leq n} \abs{z_j - z_i}^{-\alpha_j}}  \\
& \quad +  I((z_1,\alpha_1), \cdots, (z_n,\alpha_n):(f, d_{min}/2,R)).
\end{align*}
Now by the weighted AM-GM inequality we have
\begin{align*}
\frac{1}{\abs{s-z_1}^{\alpha_1 + \alpha_2}} + \frac{1}{\abs{s-z_2}^{\alpha_1 + \alpha_2}} \gtrsim_{\alpha_1, \alpha_2} \frac{1}{\abs{s-z_1}^{\alpha_1}\abs{s- z_2}^{\alpha_2} }.
\end{align*}
Hence 
\begin{align*}
& I((z_1,\alpha_1), \cdots, (z_n,\alpha_n):(f, d_{min}/2,R)) \\
&  \lesssim_{\alpha_1, \cdots, \alpha_n} I((z_1,\alpha_1 + \alpha_2), (z_3,\alpha_3), \cdots, (z_n,\alpha_n):(f, d_{min}/2,R)) \\
& \quad + I((z_1,0),(z_2,\alpha_1 + \alpha_2), (z_3,\alpha_3), \cdots, (z_n,\alpha_n):(f, d_{min}/2,R)).
\end{align*}
Now we observe that 
\begin{align*}
\frac{1}{\abs{s-z_2}} \leq \frac{3}{\abs{s - z_1}} \quad \tx{ for all } s\in   B(z_2, \abs{z_1 - z_2}/2)^c
\end{align*}
and as $d_{min} = \abs{z_1 - z_2}$ we obtain
\begin{align*}
& I((z_1,0), (z_2,\alpha_1 + \alpha_2), (z_3,\alpha_3), \cdots, (z_n,\alpha_n):(f, d_{min}/2,R)) \\
&   \lesssim_{\alpha_1, \cdots, \alpha_n} I((z_1,\alpha_1 + \alpha_2), (z_3,\alpha_3), \cdots, (z_n,\alpha_n):(f, d_{min}/2,R)).
\end{align*}
Hence proved. 
\end{proof}

\begin{lemma}\label{lem:Itwo}
Let $z_1,z_2 \in \Csp$ be such that $z_1\neq z_2$ and let $f \in \Lone(\Csp)\cap\Linfty(\Csp)$. Then
\begin{align*}
\abs{z_1 - z_2}I((z_1,1),(z_2,1):(f,0,\infty)) \lesssim \norm[\Lone\cap\Linfty]{f}\phi(\abs{z_1-z_2}).
\end{align*}
\end{lemma}
\begin{proof}
We see from \lemref{lem:Iest} that
\begin{align*}
 I((z_1,1),(z_2,1):(f,0,\infty)) & \lesssim \norm[\infty]{f} + I((z_1,2):(f,\abs{z_1 - z_2}/2,\infty))  \\
& \lesssim \norm[\infty]{f} + I((z_1,2):(f,\abs{z_1 - z_2}/2,1)) + I((z_1,2):(f,1,\infty)) \\
& \lesssim \norm[\infty]{f}\max\cbrac{-\ln(\abs{z_1 - z_2}), 1} + \norm[1]{f} \\
&  \lesssim \norm[\Lone\cap\Linfty]{f}\max\cbrac{-\ln(\abs{z_1 - z_2}), 1}.
\end{align*}
\end{proof}

\begin{lemma}\label{lem:Ithree}
Let $z_1,z_2 \in \Csp$ be non-zero complex numbers with $z_1 \neq z_2$ and let $f \in \Lone(\Csp)\cap\Linfty(\Csp)$. If $0 < \nu < 1$ then
\begin{align*}
& \abs{z_1 - z_2}I((0,1-\nu), (z_1,1), (z_2,1): (f,0,\infty)) \\
& \lesssim_{\nu} \norm[\Lone\cap\Linfty]{f} \min\cbrac*[\big]{\abs{z_1}^{\nu-1}, \abs{z_2}^{\nu-1}} \phi(\abs{z_1-z_2}).
\end{align*}
We also have the estimate
\begin{align*}
& \abs{z_1 - z_2}I((0,1-\nu), (z_1,1), (z_2,1): (f,0,\infty)) \\
& \lesssim_{\nu} \norm[\Linfty]{f} \brac{1+ \min\cbrac*[\big]{\abs{z_1}^{\nu-1}, \abs{z_2}^{\nu-1}}} \phi(\abs{z_1-z_2}).
\end{align*}
\end{lemma}
\begin{proof}
Let $d_{min} = \min\cbrac{\abs{z_1}, \abs{z_2}, \abs{z_1 - z_2}} > 0$. We prove this in cases.

\noindent \textbf{Case 1:} $d_{min} =  \min\cbrac{\abs{z_1}, \abs{z_2}}$

Without loss of generality we can assume that $d_{min} = \abs{z_1}$. Hence $\frac{\abs{z_2}}{2}\leq \abs{z_1 - z_2} \leq 2\abs{z_2}$ and so by the weighted AM-GM inequality and \lemref{lem:Iest} we have
\begin{align*}
& I((0,1-\nu), (z_1,1), (z_2,1): (f, 0,\infty)) \\
& \lesssim_{\nu} I((0,2-\nu),  (z_2,1): (f, 0,\infty)) + I((z_1,2-\nu), (z_2,1): (f, 0,\infty)) \\
& \lesssim_{\nu} \norm[\infty]{f}\abs{z_2}^{\nu - 1} + I((0,3-\nu): (f, \abs{z_2}/2,\infty)) \\
& \quad  + \norm[\infty]{f}\abs{z_1 - z_2}^{\nu - 1} +   I((z_1,3-\nu): (f, \abs{z_1 - z_2}/2,\infty)) \\
& \lesssim_{\nu} \norm[\infty]{f}\abs{z_2}^{\nu - 1} + \norm[\infty]{f}\abs{z_1 - z_2}^{\nu - 1} \\
& \lesssim_{\nu}  \norm[\Linfty]{f} \min\cbrac{\abs{z_1}^{\nu-1}, \abs{z_2}^{\nu-1}}.
\end{align*}

\noindent \textbf{Case 2:} $d_{min} =  \abs{z_1 - z_2}$

In this case we see that $\frac{\abs{z_1}}{2} \leq \abs{z_2} \leq 2\abs{z_1}$. Hence by  \lemref{lem:Iest} we have
\begin{align*}
& I((0,1-\nu), (z_1,1), (z_2,1): (f,0,\infty)) \\
&  \lesssim_{\nu} \norm[\infty]{f}\brac{\abs{z_1 - z_2}^{\nu + 1}\abs{z_1}^{-2} + \abs{z_1}^{\nu - 1}} +  I((0,1-\nu), (z_1,2): (f, \abs{z_1 - z_2}/2,\infty)) \\
& \lesssim_\nu   \norm[\infty]{f}\abs{z_1}^{\nu-1}  +  I((0,1-\nu), (z_1,2): (f, \abs{z_1 - z_2}/2, \abs{z_1}/2)) \\
& \quad  +I((0,1-\nu), (z_1,2): (f, \abs{z_1}/2, \infty)).
\end{align*}
Now observe that from the weighted AM-GM inequality we have
\begin{align*}
& I((0,1-\nu), (z_1,2): (f, \abs{z_1}/2, \infty)) \\
& \lesssim_\nu I((0,3-\nu): (f, \abs{z_1}/2, \infty)) + I((z_1,3-\nu): (f, \abs{z_1}/2, \infty)) \\
& \lesssim_\nu  \norm[\infty]{f}\abs{z_1}^{\nu-1}.
\end{align*} 
Hence we have
\begin{align*}
& I((0,1-\nu), (z_1,1), (z_2,1): (f,0,\infty)) \\
& \lesssim_\nu   \norm[\infty]{f}\abs{z_1}^{\nu-1} +  I((0,1-\nu), (z_1,2): (f, \abs{z_1 - z_2}/2, \abs{z_1}/2)) \\
&  \lesssim_\nu    \norm[\infty]{f}\abs{z_1}^{\nu-1} + \abs{z_1}^{-2}I((0,1-\nu):(f, \abs{z_1 - z_2}/2, \abs{z_1}/2)) \\
& \quad + \abs{z_1}^{\nu-1}I((z_1,2):(f,\abs{z_1 - z_2}/2, \abs{z_1}/2)) \\
&  \lesssim_\nu    \norm[\infty]{f}\abs{z_1}^{\nu-1} +  \abs{z_1}^{\nu-1}I((z_1,2):(f,\abs{z_1 - z_2}/2, \abs{z_1}/2)) \\
&  \lesssim_\nu  \norm[\infty]{f}\abs{z_1}^{\nu-1}  + \abs{z_1}^{\nu-1}I((z_1,2):(f,\abs{z_1 - z_2}/2, 1)) + \abs{z_1}^{\nu-1}I((z_1,2):(f,1, \abs{z_1}/2))  \\
& \lesssim_{\nu}  \norm[\infty]{f} \min\cbrac{\abs{z_1}^{\nu-1}, \abs{z_2}^{\nu-1}} \max\cbrac{-\ln \abs{z_1 - z_2}, 1} + \abs{z_1}^{\nu-1}I((z_1,2):(f,1, \abs{z_1}/2)).
\end{align*}
We now easily see that
\begin{align*}
\abs{z_1}^{\nu-1}I((z_1,2):(f,1, \abs{z_1}/2)) \lesssim_{\nu} \abs{z_1}^{\nu-1}\norm[1]{f}.
\end{align*}
Now $I((z_1,2):(f,1, \abs{z_1}/2)) $ is non-zero only if $\abs{z_1}\geq 2$ and that $\abs{z_1}^{\nu-1}\ln(\abs{z_1}/2) \lesssim_\nu 1$ if $\abs{z_1} \geq 2$. Therefore
\begin{align*}
\abs{z_1}^{\nu-1}I((z_1,2):(f,1, \abs{z_1}/2)) \lesssim_{\nu} \norm[\infty]{f}.
\end{align*}
Hence proved. 
\end{proof}


\bibliographystyle{amsplain}
\bibliography{Mainv2.bib}

\end{document}